\documentclass{amsart}
\usepackage[all]{xy}

\usepackage{amsmath, amsfonts, amssymb, enumerate}
\SelectTips{cm}{10}\UseTips

\newcommand{\mB}{\mathcal{B}}

\newcommand{\mE}{\mathcal{E}}
\newcommand{\mF}{\mathcal{F}}

\newcommand{\mL}{\mathcal{L}}

\newcommand{\mN}{\mathcal{N}}

\newcommand{\mP}{\mathcal{P}}

\newcommand{\mR}{\mathcal{R}}
\newcommand{\mS}{\mathcal{S}}
\newcommand{\mT}{\mathcal{T}}

\newcommand{\fm}{\mathfrak{m}}
\newcommand{\fM}{\mathfrak{M}}

\newcommand{\fp}{\mathfrak{p}}

\newcommand{\fq}{\mathfrak{q}}

\newcommand{\fT}{\mathfrak{T}}

\newcommand{\bfA}{\mathbf{A}}

\newcommand{\bfC}{\mathbf{C}}

\newcommand{\bfF}{\mathbf{F}}

\newcommand{\bfN}{\mathbf{N}}

\newcommand{\bfP}{\mathbf{P}}
\newcommand{\bfQ}{\mathbf{Q}}

\newcommand{\bfT}{\mathbf{T}}

\newcommand{\bfZ}{\mathbf{Z}}

\newcommand{\Oo}{\mathcal{O}}

\newcommand{\AF}{\mathbf{A}_F}

\newcommand{\AQ}{\mathbf{A}}

\newcommand{\OFv}{\mathcal{O}_{F,v}}

\newcommand{\tuint}{\textup{int}}

\newcommand{\ov}{\overline}

\newcommand{\be}{\begin{equation}}
\newcommand{\ee}{\end{equation}}
\newcommand{\bes}{\begin{equation*}}
\newcommand{\ees}{\end{equation*}}

\newcommand{\bs}{\begin{split}}
\newcommand{\es}{\end{split}}
\newcommand{\bss}{\begin{split*}}
\newcommand{\ess}{\end{split*}}

\newcommand{\bmat}{\left[ \begin{matrix}}
\newcommand{\emat}{\end{matrix} \right]}
\newcommand{\bsmat}{\left[ \begin{smallmatrix}}
\newcommand{\esmat}{\end{smallmatrix} \right]}

\newcommand{\bml}{\begin{multline}}
\newcommand{\eml}{\end{multline}}
\newcommand{\bmls}{\begin{multline*}}
\newcommand{\emls}{\end{multline*}}

\DeclareMathOperator{\Cl}{Cl}

\DeclareMathOperator{\End}{End}
\DeclareMathOperator{\Ext}{Ext}

\DeclareMathOperator{\Frob}{Frob}
\DeclareMathOperator{\Gal}{Gal}
\DeclareMathOperator{\GL}{GL}

\DeclareMathOperator{\Hom}{Hom}

\DeclareMathOperator{\image}{Im}

\DeclareMathOperator{\Res}{Res}

\DeclareMathOperator{\Spec}{Spec}

\DeclareMathOperator{\val}{val}

\newcommand{\iy}{\infty}

\newcommand{\tr}{\textup{tr}\hspace{2pt}}

\usepackage[all]{xy}
\usepackage{amsthm}
\usepackage{amssymb, amsfonts}
\SelectTips{cm}{10}\UseTips
\theoremstyle{plain}
\newtheorem{thm}{Theorem}
\newtheorem{prop}[thm]{Proposition}
\newtheorem{cor}[thm]{Corollary}
\newtheorem{lemma}[thm]{Lemma}
\newtheorem{conj}[thm]{Conjecture}

\theoremstyle{definition}
\newtheorem{definition}[thm]{Definition}
\newtheorem{notation}[thm]{Notation}

\newtheorem{example}[thm]{Example}
\newtheorem{rem}[thm]{Remark} 
\newtheorem{assumption}[thm]{Assumption}

\numberwithin{thm}{section}
\numberwithin{equation}{section}

\DeclareMathOperator{\Fit}{Fitt}

\numberwithin{equation}{section}

\begin{document}
\author{Tobias Berger$^1$ \and
Krzysztof Klosin$^2$}
\address{$^1$School of Mathematics and Statistics, University of Sheffield, Hicks Building, Hounsfield Road, Sheffield S3 7RH, UK, email: tberger@cantab.net}
\address{$^2$Department of Mathematics,
Queens College,
City University of New York,
65-30 Kissena Blvd,
Queens, NY 11367, USA, email: kklosin@qc.cuny.edu}
\title[On modularity of reducible residual Galois representations]{On lifting and modularity of reducible residual Galois representations over imaginary quadratic fields}


\thanks{The work of the second author was  partially supported by a PSC-CUNY Award, jointly funded by The Professional
Staff Congress and The City University of New York.}
\maketitle 













\begin{abstract}
In this paper we study deformations of mod $p$ Galois representations $\tau$ (over an imaginary quadratic field $F$)
of dimension $2$ whose semi-simplification is the direct sum of two characters $\tau_1$ and $\tau_2$. As opposed to \cite{BergerKlosin13} we do not impose any restrictions on the dimension of the crystalline Selmer group $H^1_{\Sigma}(F, {\rm Hom}(\tau_2, \tau_1)) \subset {\rm Ext}^1(\tau_2, \tau_1)$. 
We establish that there
exists a basis $\mB$ of $H^1_{\Sigma}(F, {\rm Hom}(\tau_2, \tau_1))$ arising from automorphic representations over $F$ (Theorem \ref{mainthm}).
Assuming among other things that the elements of $\mB$ admit only finitely many crystalline characteristic 0 deformations we prove a modularity lifting theorem asserting that if $\tau$ itself is modular then so is its every crystalline characteristic zero deformation  (Theorems \ref{mainthm2} and \ref{version2}). 
\end{abstract}

\maketitle

\section{Introduction} 
Let $p$ be an odd prime. Let $F$ be a number field, $\Sigma$ a finite set of primes of $F$ (containing all primes $\fp$ of $F$ lying over $p$) and $G_{\Sigma}$ the Galois group of the maximal extension
of $F$ unramified outside $\Sigma$. Let $E$ be a finite extension of $\bfQ_p$ with ring of integers $\Oo$ and residue field $\bfF$. Let  $\tau_1, \tau_2: G_{\Sigma} \rightarrow \GL_{n_i}(\bfF)$ be two absolutely irreducible non-isomorphic representations
 with $n_1+n_2=n$, which we assume lift uniquely to  crystalline representations $\tilde \tau_i: G_{\Sigma} \rightarrow \GL_{n_i}(\Oo)$.

The aim of this article is to study deformations of \emph{non-semi-simple} continuous crystalline representations
$\tau:
G_{\Sigma} \rightarrow \GL_n(\bfF)$ whose semi-simplification is $\tau_1 \oplus \tau_2$ in the case $n=2$ and $F$ is an imaginary quadratic field. We analyzed this deformation problem in \cite{BergerKlosin13} under the additional assumption that $H^1_{\Sigma}(F, \Hom(\tau_2, \tau_1))$ is one-dimensional (which is equivalent to saying that there exists only one such $\tau$ up to isomorphism). Here $H^1_{\Sigma}$
denotes the subgroup of $H^1$ consisting of classes unramified outside $\Sigma$ and crystalline at all $\fp \mid p$. In this paper we do not make any assumption on this dimension. Disposing of the ``dim=1'' assumption is more than a technicality as in the general case one can no longer expect to be able to identify the universal deformation ring with a Hecke algebra.

This question was studied by Skinner and Wiles for $n=2$ and totally real fields $F$ in  the seminal paper \cite{SkinnerWiles99}. In that paper the authors analyze primes $\fq$ of the (ordinary) universal deformation ring $R_{\tau}$ of $\tau$ and prove that they are `pro-modular' in the sense that the trace of the deformation corresponding to $R_{\tau}\twoheadrightarrow R_{\tau}/\fq$ occurs in the Hecke algebra $\bfT$. In particular no direct identification of $R_{\tau}$ and $\bfT$ is made.

In this article we take a different approach and work with the reduced universal deformation ring $R_{\tau}^{\rm red}$ and its ideal of reducibility. In the ``dim=1''-case, the authors proved (as a consequence of an $R=T$-theorem - Theorem 9.14 in \cite{BergerKlosin13}) that $R_{\tau}^{\rm red}$ is a finitely generated $\bfZ_p$-module.  In contrast, if $\dim H^1_{\Sigma}(F, \Hom(\tau_2, \tau_1))>1$, while  there are only finitely many automorphic representations whose associated Galois representations are deformations of $\tau$, the ring $R^{\rm red}_{\tau}$ may potentially be infinite over $\bfZ_p$ (Remark \ref{baddef}). This is a direct consequence of  the existence of linearly independent cohomology classes inside the Selmer group which can be used to construct non-trivial lifts to $\GL_2(\bfF[[X]])$. 
The resulting (potentially large)  characteristic $p$ components of $R^{\rm red}$ do not arise from automorphic representations and in this paper we will ignore them by considering 
a certain torsion-free quotient $R_{\tau}^0$ of $R^{\rm red}_{\tau}$ 
 instead of $R_{\tau}^{\rm red}$ itself. It is however possible that by doing so we are excluding some characteristic $p$ deformations whose traces may be modular in the sense of \cite{CalegariMazur09} (i.e. arise from torsion Betti cohomology classes).

On the other hand, as opposed to the situation studied in \cite{SkinnerWiles99}, over an imaginary quadratic field  there are no reducible deformations to characteristic zero which in turn  is a consequence of the finiteness of the Bloch-Kato Selmer group $ H^1_{\Sigma}(F, \Hom(\tilde{\tau}_2, \tilde{\tau}_1)\otimes\bfQ_p/\bfZ_p)$ (Lemma \ref{irr4}), where $\tilde{\tau}_1$, $\tilde{\tau}_2$ are (unique) lifts  to characteristic zero of $\tau_1$ and $\tau_2$ respectively.

While each $\tau$ may possess non-modular reducible characteristic $p$ deformations, 
the situation is complicated further by the fact that in general many $\tau$'s do not admit any modular deformations at all (this phenomenon does not arise in the ``dim=1'' case).
Indeed, first note   that two extensions in $\Ext_{G_{\Sigma}}^1(\tau_2, \tau_1)$ define isomorphic representation of $G_{\Sigma}$ if and only if they are (non-zero) scalar multiples of each other. In particular, if $\dim_{\bfF} \Ext_{G_{\Sigma}}^1(\tau_2, \tau_1)=1$, then there is a unique non-semi-simple representation of $G_{\Sigma}$ with semi-simplification $\tau_1 \oplus \tau_2$. (Similarly, if  $\dim_{\bfF} H^1_{\Sigma}(F, \Hom(\tau_2, \tau_1))=1$ then there exists a unique crystalline such representation.) However, if  $\dim_{\bfF}\Ext_{G_{\Sigma}}^1(\tau_2, \tau_1)=m$, then there are $\frac{q^m-1}{q-1}$ non-isomorphic such representations where $q=\# \bfF$. This demonstrates  that in general not all reducible representations $\tau$ can be modular (of a particular level and weight), as the number of such characteristic zero automorphic forms is fixed (in particular it is independent of making a residue field extension).
 Nevertheless, we are able to prove (see Corollary \ref{surj1prime}) that there
exists an $\bfF$-basis $\mB$ of $H^1_{\Sigma}(F, {\rm Hom}(\tau_2, \tau_1))$ arising from modular forms. For this we combine a congruence ideal bound for a Hecke algebra with the upper bound on the Selmer group
of ${\rm Hom}(\tilde{\tau}_2, \tilde{\tau}_1)$ predicted by the Bloch-Kato conjectures.

Let $R^{\rm tr, 0}_{\tau}$ be the image in $R_{\tau}^0$ of the subalgebra generated by traces of $R_{\tau}^{\rm red}$
for $\tau$ arising from a modular form. As pointed out we can extend the set consisting of $\tau$ to a modular basis $\mB:=\{\tau^1=\tau, \tau^2, \dots, \tau^s\}$ of $H^1_{\Sigma}(F, {\rm Hom}(\tau_2, \tau_1))$. Our ultimate goal is to show that it is possible to identify $R^{\rm tr, 0}_{\tau}$ with the quotient $\bfT_{\tau}$ of a Hecke algebra $\bfT$.  
Here the quotient $\bfT_{\tau}$ corresponds to automorphic forms for which there exists a lattice in the associated Galois representation with respect to which the mod $p$ reduction equals $\tau$.

To prove our main modularity lifting theorem (Theorem \ref{mainthm2}) we work under the following two assumptions. On the one hand we assume that the modular basis $\mB$ is unique in the sense that any other such consists of scalar multiples of the elements of $\mB$. On the other hand we assume that   all $\tau \in \mB$ admit only finitely many characteristic zero deformations, which in particular implies that the quotient $R^0_{\tau}$ we define is a finitely generated $\bfZ_p$-module.
The first assumption can be replaced with the assumption that the Bloch-Kato Selmer group $ H^1_{\Sigma}(F, \Hom(\tilde{\tau}_2, \tilde{\tau}_1)\otimes\bfQ_p/\bfZ_p)$ is annihilated by $p$ (Theorem \ref{version2}). This second result is in a sense `orthogonal' to the main results of \cite{BergerKlosin11} and \cite{BergerKlosin13}, where the same Selmer group is assumed to be cyclic, but of arbitrary finite order. 

Our approach relies on simultaneously considering all the deformation problems for representations $\tau^i$ ($i=1,2,\dots, s$). 
 As in \cite{BergerKlosin13} we first study ``reducible'' deformations via the quotients $R^{\rm tr, 0}_{\tau^i}/I^{\rm tr, 0}_{\tau^i}$ for the reducibility ideal $I^{\rm tr, 0}_{\tau^i}$ of the trace of the universal deformation into $\GL_2(R_{\tau}^0)$ as defined by Bella\"iche and Chenevier. These ideals are the analogues of Eisenstein ideals $J_{\tau_i}$ on the Hecke algebra side. To relate $\# \prod_i R^{\rm tr, 0}_{\tau^i}/I^{\rm tr, 0}_{\tau^i}$ to the order of a Bloch-Kato Selmer group we make use of a lattice construction of Urban (Theorem 1.1 of \cite{Urban01}, see Theorem \ref{Urban 1.1} in this paper). 
In fact it is a repeated application of Urban's theorem (on the Hecke side and on the deformation side) that allows us to prove a modularity lifting theorem.
We show that 
when the upper bound on the Selmer group and the lower bound on the congruence ideal agree (which in many cases is a consequence of the Bloch-Kato conjecture), this implies that every reducible deformation which lifts to characteristic zero of every $\tau^i$ is modular (cf. section \ref{urban}). 
It is here that we make use of the assumption on the `uniqueness' of $\mB$ to be able to use a result of Kenneth Kramer and the authors \cite{BergerKlosinKramer14} on the distribution of Eisenstein-type congruences among various residual isomorphism classes of Galois representations (cf. Section \ref{Tsection}). 
Yet another application of Urban's Theorem allows us to prove the existence of a deformation to $\GL_2(R_{\tau}^{\rm tr, 0})$ and as a consequence to identify $R_{\tau}^{\rm tr, 0}$ with $R_{\tau}^0$ (Theorem \ref{deformtoRtr}). Using the fact that the ideal of reducibility 
of $R_{\tau}^0$ is principal (Proposition \ref{prin}) and applying the commutative algebra criterion (Theorem 4.1 in  \cite{BergerKlosin13}) we are finally able to obtain an isomorphism $R_{\tau}^{\rm red} \cong \bfT_{\tau}$ and thus a modularity lifting theorem (Theorems \ref{mainthm2} and \ref{version2}).

Throughout the paper we work in a slightly greater generality than necessary for the imaginary quadratic case to stress that our results apply in a more general context if one assumes some standard conjectures. However,  in section \ref{Main result} we gather all the assumptions in the imaginary quadratic case as well as the statements of the main theorems (Theorems \ref{mainthm}, \ref{mainthm2} and \ref{version2}) in this context. Hence the reader may refer directly to that section for the precise (self-contained) statements of the main results of the paper in that case.

We would like to thank Gebhard B\"ockle and Jack Thorne for helpful comments and conversations related to the contents of this article.  We would also like to express our gratitude to the anonymous referee for suggesting numerous improvements throughout the article.  The second author was partially supported by a PSC-CUNY Award, jointly funded by The Professional Staff Congress and The City University of New York.

\section{Deformation rings} 
Let $F$ be a number field and $p>2$ a prime with $p \nmid \# \Cl_F$  and $p$ unramified in $F/\bfQ$. Let $\Sigma$ be a finite set of finite places of $F$ containing all the places
lying over $p$. Let $G_{\Sigma}$ denote the Galois
group $\Gal(F_{\Sigma}/F)$, where $F_{\Sigma}$ is the maximal extension of $F$ unramified outside $\Sigma$.  For every prime $\fq$ of $F$ we fix compatible embeddings $\ov{F} \hookrightarrow \ov{F}_{\fq} \hookrightarrow \bfC$ and write $D_{\fq}$ and $I_{\fq}$ for the corresponding decomposition and inertia subgroups of $G_F$ (and also their images in $G_{\Sigma}$ by a slight abuse of notation). Let $E$ be a (sufficiently large) finite extension of $\bfQ_p$ with ring of integers $\Oo$ and residue field $\bfF$. We fix a choice
of a  uniformizer $\varpi$. 

\subsection{Deformations} \label{s2.1}
 Denote
the category of local complete Noetherian $\Oo$-algebras with residue field $\bfF$ by $\textup{LCN}(E)$. 
Let $m$ be any positive integer. Suppose $$r: G_{\Sigma} \to {\rm GL}_m(\bfF)$$ is a continuous homomorphism.

We recall from \cite{ClozelHarrisTaylor08} p. 35 the definition of a \emph{crystalline} representation: Let $\fp \mid p$
and $A$ be a complete Noetherian $\bfZ_p$-algebra. A representation $\rho: D_{\fp} \to {\rm GL}_n(A)$ is crystalline if for each Artinian quotient
$A'$ of $A$, $\rho \otimes A'$ lies in the essential image of the Fontaine-Lafaille functor $\mathbf{G}$ (for its definition see e.g. \cite{BergerKlosin13} Section 5.2.1). We also call a continuous finite-dimensional $G_{\Sigma}$-representation $V$ over $\bfQ_p$  (short) \emph{crystalline} if, for all primes $\fp \mid p$,
${\rm Fil}^0 D=D$ and ${\rm Fil}^{p-1} D=(0)$ for the filtered vector space $D=(B_{\rm crys} \otimes_{\bfQ_p} V)^{D_{\fp}}$ defined by Fontaine (for details see \cite{BergerKlosin13} Section 5.2.1).

Following Mazur we call two representations $\tilde r_1, \tilde r_2: G_{\Sigma} \rightarrow \GL_{m}(A)$ for $A \in \textup{LCN}(E)$ such that $r = \tilde r_1=\tilde r_2 \pmod{\fm_A}$ \emph{strictly equivalent} if there exists $M \in {\rm ker}({\rm GL}_2(A) \to {\rm GL}_2(\bfF))$ such that $\tilde r_1=M \tilde r_2 M^{-1}$. A (crystalline) \emph{$\Oo$-deformation} of $r$ is then a pair
consisting of $A \in \textup{LCN}(E)$ and a strict equivalence class of continuous representations $\tilde r: G_{\Sigma} \rightarrow \GL_{m}(A)$ that are crystalline 
at the primes dividing $p$ 
and
such that $r = \tilde r \pmod{\fm_A}$, where $\fm_A$ is the maximal ideal of $A$. (So, in particular we do not impose on
our lifts any conditions at primes in $\Sigma \setminus \Sigma_p$.) 
 Later we assume that if $\fq \in \Sigma$, then $\# k_{\fq} \not\equiv 1$ (mod $p$), which means that all deformations we consider will trivially be ``$\Sigma$-minimal".
As is customary we will denote a deformation by a single member of
its strict equivalence class.

If $r$ has a scalar centralizer then the deformation functor is representable by $R_{r} \in \textup{LCN}(E)$ since crystallinity is a deformation condition in the sense of \cite{Mazur97}. We denote the universal crystalline $\Oo$-deformation by $\rho_{r} :
G_{\Sigma} \rightarrow \GL_{m}(R_{r})$. Then for every $A \in \textup{LCN}(E)$ there is a one-to-one correspondence between the set
of $\Oo$-algebra maps $R_{r} \rightarrow A$  and the set of crystalline deformations $\tilde r: G_{\Sigma}
\rightarrow \GL_{m}(A)$ of $r$.

For $j\in \{1,2\}$ let $\tau_j: G_{\Sigma} \to \GL_{n_j}(\bfF)$ be an absolutely irreducible continuous representation. Assume that $\tau_1 \not\cong \tau_2$. Consider the set of isomorphism classes of $n$-dimensional residual (crystalline at all primes $\fp \mid p$) representations of the form: \be \label{form1}\tau=\bmat \tau_1 & * \\ & \tau_2\emat: G_{\Sigma} \rightarrow \GL_n(\bfF),\ee which are non-semi-simple ($n=n_1+n_2$).

 From
now on assume $p \nmid n!$.

\begin{lemma} Every representation $\tau$ of the form (\ref{form1}) has scalar centralizer. \end{lemma}
\begin{proof} This is easy.  \end{proof}

\subsection{Pseudo-representations and pseudo-deformations}

We next recall the notion of a pseudo-representation (or pseudo-character) and pseudo-deformations (from \cite{BellaicheChenevierbook} Section 1.2.1 and \cite{Boeckle11} Definition 2.2.2). 

\begin{definition} Let $A$ be a topological ring and $R$ a topological $A$-algebra. A (continuous) $A$-valued pseudo-representation on $R$ of dimension $d$,
for some $d\in \bfN_{>0}$, is a continuous function $T : R \to A$ such that

\begin{enumerate}[(i)]
\item $T(1)=d$ and $d!$ is a non-zero divisor of $A$;
\item $T$ is \emph{central}, i.e. such that $T(xy)=T(yx)$ for all $x,y \in R$;

\item $d$ is minimal such that $S_{d+1}(T)(x)=0$, where, for every integer $N \geq 1$, $S_{N}(T): R^{N} \to A$ is given by $$S_{N}(T)(x)= \sum_{\sigma \in \mathcal{S}_{N}} \epsilon(\sigma) T^{\sigma}(x),$$ where for a cycle $\sigma=(j_1, \ldots j_m)$ we define $T^{\sigma}((x_1, \ldots x_{d+1}))=T(x_{j_1} \cdots x_{j_m})$, and for a general permutation $\sigma$ with cycle decomposition $\prod_{i=1}^r \sigma_i$ we let $T^{\sigma}(x)=\prod_{i=1}^r T^{\sigma_i}(x)$. 

\end{enumerate}
In the case when $R=A[G_{\Sigma}]$ the pseudo-representation $T$ is determined by its restriction to $G_{\Sigma}$  (see \cite{BellaicheChenevierbook} Section 1.2.1)  and we will also call the restriction of $T$ to $G_{\Sigma}$ a pseudo-representation. \end{definition}

We note that if $\rho: A[G_{\Sigma}] \to M_n(A)$ is a morphism of $A$-algebras then $\tr \rho$ is a pseudo-representation of dimension $n$ (see \cite{BellaicheChenevierbook} Section 1.2.2).

According to \cite{BellaicheChenevierbook} Section 1.2.1, if $T: R \to A$ is a pseudo-representation of dimension $d$ and $A'$ an $A$-algebra, then $T \otimes A': R \otimes A' \to A'$ is again a pseudo-representation of dimension $d$.

Following \cite{SkinnerWiles99} (see also \cite{Boeckle11} Section 2.3) we define a pseudo-deformation of $\tr \tau_1 + \tr \tau_2$ to be a pair $(T,A)$
consisting of $A \in \textup{LCN}(E)$ and a continuous pseudo-representation $T:G_{\Sigma} \to A$
such that $T = \tr \tau_1 + \tr \tau_2 \pmod{\fm_A}$, where $\fm_A$ is the maximal ideal of $A$.

By the sentence following \cite{SkinnerWiles99} Lemma 2.10 (see also \cite{Boeckle11} Proposition 2.3.1) there exists a universal pseudo-deformation ring $R^{\rm ps} \in {\rm LCN}(E)$ and we write $T^{\rm ps}: G_{\Sigma} \to R^{\rm ps}$ for the universal pseudo-deformation.
For every $A \in \textup{LCN}(E)$ there is a one-to-one correspondence between the set
of $\Oo$-algebra maps $R^{\rm ps} \rightarrow A$ and the set of pseudo-deformations $T:G_{\Sigma} \to A$ of $\tr \tau_1 + \tr \tau_2$. Any deformation of a representation $\tau$ as in (\ref{form1}) gives rise (via its trace) to a pseudo-deformation of $\tr \tau_1 + \tr \tau_2$, so there exists a unique $\Oo$-algebra map $R^{\rm ps} \to R_{\tau}$ such that  the trace of the deformation equals the composition of $T^{\rm ps}$ with $R^{\rm ps} \to R_{\tau}$.

We write $R_{\tau}^{\rm red}$ for the quotient of $R_{\tau}$ by its nilradical and $\rho_{\tau}^{\rm red}$ for the corresponding universal deformation, i.e. the composite of $\rho_{\tau}$ with $R_{\tau} \twoheadrightarrow R_{\tau}^{\rm red}$. We further write $R_{\tau}^{\rm tr}\subset R_{\tau}^{\rm red}$ for the closed $\Oo$-subalgebra of $R_{\tau}^{\rm red}$ generated by the set $$S:=\{\tr \rho_{\tau} (\Frob_{\fq}) \mid \fq \not\in \Sigma\}.$$

\begin{lemma} \label{2} The image of $R^{\rm ps} \to R_{\tau}^{\rm red}$ is $R^{\rm tr}_{\tau}$  and hence $R_{\tau}^{\rm tr}$ is an object  in the category $\textup{LCN}(E)$.. \end{lemma} 

\begin{proof} This is clear (cf. \cite{ChoVatsal03} Theorem 3.11) since $R^{\rm ps}$ is topologically generated by $T(\Frob_{\fp}$) (and $R_{\tau}^{\rm tr}$ is closed). 
\end{proof}

\subsection{Selmer groups} For a crystalline $p$-adic $G_{\Sigma}$-module $M$ (finitely generated or cofinitely generated over $\Oo$ - for precise definitions cf. \cite{BergerKlosin13}, section 5) we define the Selmer group $H^1_{\Sigma}(F, M)$ to be the subgroup of $H^1_{\rm cont}(F_{\Sigma}, M)$ consisting of cohomology classes which are crystalline at all primes $\fp$ of $F$ dividing $p$. Note that we place no restrictions at the primes in $\Sigma$ that do not lie over $p$. For more details cf. [loc.cit.]. 

We are now going to state our assumptions. The role of the first one is to rigidify the problem of deforming the representations $\tau_j$ appearing on the diagonal of the residual representations. The role of the second is to rule out characteristic zero upper triangular deformations.
\begin{assumption} \label{ass1} Assume that $R_{\tau_j}=\Oo$ and denote by $\tilde{\tau}_j$ the unique lifts of $\tau_j$ to $\GL_{n_j}(\Oo)$. \end{assumption}

\begin{assumption} [``Bloch-Kato conjecture''] \label{BK} One has the following bound:
$$\# H^1_{\Sigma}(F, \Hom_{\Oo}(\tilde{\tau}_2, \tilde{\tau}_1)\otimes_{\Oo}E/\Oo) \leq \# \Oo/L,$$ for some non-zero $L \in \Oo$. \end{assumption}
\begin{rem} In applications the constant $L$ will be the special $L$-value at zero of the Galois representation $\Hom_{\Oo}(\tilde{\tau}_2, \tilde{\tau}_1)$ divided by an appropriate period. \end{rem}

For the remainder of this section we will work under the above two assumptions.

\subsection{Ideal of reducibility}
Let $A$ be a Noetherian Henselian local (commutative) ring with maximal ideal $\fm_A$ and residue field $\bfF$ and let $R$ be an $A$-algebra. 
We recall from \cite{BellaicheChenevierbook} Proposition 1.5.1 the definition of the ideal of reducibility of a (residually multiplicity free) pseudo-representation $T:R \to A$ of dimension $n$, for which we assume that
$$T = \tr \tau_1 + \tr \tau_2 \mod{\fm_A}$$ 

\begin{definition} [\cite{BellaicheChenevierbook} Proposition 1.5.1 and Definition 1.5.2]
	There exists a smallest ideal $I$ of $A$ such that $T$ mod $I$ is the sum of two pseudo-characters $T_1,
T_2$ with $T_i = \tr \tau_i$ mod $\fm_A$. We call this smallest ideal the \emph{ideal of reducibility} of $T$ and denote it by $I_T$. \end{definition}

\begin{definition}
 We will write $I^{\rm ps} \subset R^{\rm ps}$ for the ideal of reducibility of the universal pseudo-deformation $T^{\rm ps}: R^{\rm ps}[G_{\Sigma}] \to R^{\rm ps}$, $I_{\tau} \subset R_{\tau}$ for the ideal of reducibility of $\tr \rho_{\tau}:R_{\tau}[G_{\Sigma}] \to R_{\tau}$, $I_{\tau}^{\rm red} \subset R_{\tau}^{\rm red}$ for the ideal of reducibility of $\tr \rho_{\tau}^{\rm red}:R_{\tau}^{\rm red}[G_{\Sigma}] \to R_{\tau}^{\rm red}$ and $I^{\rm tr}_{\tau}$ for the ideal of reducibility of $\tr \rho_{\tau}^{\rm red}:R^{\rm tr}_{\tau}[G_{\Sigma}] \to R^{\rm tr}_{\tau}$. 
\end{definition}

\begin{lemma} \label{l2.6}
Let $I_{0}$ be the smallest closed ideal of $R_{\tau}^{\rm tr}$ containing the set $$\{  \tr \rho_{\tau}^{\rm red}(\Frob_v) - \tr \tilde{\tau}_1(\Frob_v) - \tr \tilde{\tau}_2(\Frob_v) \mid v \not\in \Sigma \}. $$ 
Then $I_{0}$ equals the ideal of reducibility $I_{\tau}^{\rm tr}\subset R_{\tau}^{\rm tr}$.
\end{lemma}

\begin{proof}
By the Chebotarev density theorem we get $  \tr \rho_{\tau}^{\rm red}= \tr\tilde{\tau}_1 + \tr\tilde{\tau}_2$ (mod $I_{0}$), hence $I_0 \supset I_{\tau}^{\rm tr}$. 
Conversely, we know from the definition of the ideal of reducibility that $\tr \rho_{\tau}^{\rm red}$ (mod $I^{\rm tr}_{\tau}$) is given by the sum of two pseudo-characters reducing to $\tr \tau_i$. By Assumption \ref{ass1} and Theorems 7.6 
and 7.7 of \cite{BergerKlosin13}  (see also \cite{Boeckle11} Theorem 2.4.1)  these two pseudo-characters must equal $\tr \tilde{\tau}_i$ (mod $I_{\tau}^{\rm tr}$). This shows that $I_{\tau}^{\rm tr} \supset I_0$.
\end{proof}

\begin{cor}
The quotient $R_{\tau}^{\rm tr}/I_{\tau}^{\rm tr}$ is cyclic. \qed
\end{cor} 

\begin{rem} \label{r2.10}
Combined with Lemma 7.11 of \cite{BergerKlosin13} this shows that for any pseudo-deformation $T:A[G_{\Sigma}] \to A$ of $\tr \tau_1 + \tr \tau_2$  with ideal of reducibility $I_T$ for 
which there is a surjection $R_{\tau}^{\rm tr} \to A$, the quotient $A/I_T$ is cyclic.
\end{rem}

\begin{prop} \label{tor} The module $R_{\tau}/I_{\tau}$ is a torsion $\Oo$-module. \end{prop}
\begin{proof} 
Fix $\sigma \in \bfZ_+$ and set $S:=R_{\tau}/I_{\tau}$. Suppose that $S$ is not torsion. Let $\phi: S\twoheadrightarrow R:=S/\varpi^{\sigma}S$ be the canonical surjection (of $\Oo$-algebras). Let $A:= \phi(\Oo)$. We first claim that $A = \Oo/\varpi^{\sigma}\Oo$.
Clearly $\varpi^{\sigma}=0$ in $S/\varpi^{\sigma}S$, so we just need to prove that $\varpi^{\sigma-1} \not\in \varpi^{\sigma}S$. Suppose on the contrary that $\varpi^{\sigma-1} \in \varpi^{\sigma}S$. Then there exists $s \in S$ such that 
\be\label{1}\varpi^{\sigma-1}(1-\varpi s)=0\quad \textup{in $S$}.\ee 
Since the residue field
of $S$ is $\Oo/\varpi=\bfF$, we see that $\varpi$ is not a unit in $S$, and hence
$1-\varpi s$ is a unit in $S$. Thus (\ref{1}) implies that $\varpi^{\sigma-1}=0$ in $S$, which leads to a contradiction and hence we have proved that $A=\Oo/\varpi^{\sigma}\Oo$.

We now use the following lemma.
\begin{lemma} \label{complete} There exists an $\Oo$-submodule $B \subset R$ such that $$R = A \oplus B$$ as $\Oo$-modules. \end{lemma}
\begin{proof} This follows from the following result.

\begin{lemma}[Lemma 6.8(ii), p.222 in \cite{Hungerford}] \label{Hunger} Let $A'$ be a module over a PID $R'$ such that $p^nA'=0$ and $p^{n-1}A' \neq 0$ for some prime $p \in R'$ and a positive integer $n$. Let $a$ be an element of $A'$ of order $p^n$. Then there is a submodule $C'$ of $A'$ such that  $A'=R'a \oplus C'$. \end{lemma} 
Apply Lemma \ref{Hunger} for $R'=\Oo$, $A'=R$, $p=\varpi$, $n=\sigma$, $a=\psi(1)$. Then $R'a = A$. 
\end{proof}
We now finish the proof of Proposition \ref{tor}.  Let $e$ be an $\Oo$-module generator of $A$. Write $\rho_I: G \to \GL_n(R)$ for the deformation corresponding to the canonical map $R_{\tau} \twoheadrightarrow R$. Then we can write $$\rho_I = \bmat \tilde{\tau}_1 & \alpha e +\beta \\ & \tilde{\tau}_2\emat,$$ where $\alpha: G \to M_{n_1 \times n_2}(\Oo)$
and $\beta : G \to M_{n_1 \times n_2}(B)$ are maps (here we identify $\tilde{\tau}_j$ with its composition with $\Oo \to R$). Define $$\rho_I^+: G \to \GL_n(A) \quad g \mapsto \bmat \tilde{\tau}_1(g) & \alpha(g) e \\ & \tilde{\tau}_2(g)\emat.$$ We must check that $\rho_I^+$ is a homomorphism. This follows easily from the fact that $\rho_I$ is a homomorphism and the fact that $A$ is a direct summand of $R$. 
Moreover, note that the image of $\alpha$ is not contained in $M_{n_1 \times n_2}(\varpi \Oo)$ because $\rho_I$ reduces to $\tau$ which is not semi-simple.

Note that $\rho^+_I$ is an upper-triangular deformation into $\GL_n(\Oo/\varpi^{\sigma})$. Moreover, since $\rho_I^+$ reduces to $\tau$, it gives rise to an element in $H^1_{\Sigma}(F, \Hom_{\Oo}(\tilde{\tau}_2, \tilde{\tau}_1)\otimes E/\Oo)$ which generates an $\Oo$-submodule isomorphic to $\Oo/\varpi^{\sigma}$. Since $\sigma$ was arbitrary we conclude that $H^1_{\Sigma}(F, \Hom_{\Oo}(\tilde{\tau}_2, \tilde{\tau}_1)\otimes E/\Oo)$ must be infinite which contradicts Assumption \ref{BK}. This concludes the proof of Proposition \ref{tor}. 
\end{proof} 

\begin{rem} \label{baddef} If $\dim_{\bfF}H^1_{\Sigma}(F, \Hom(\tau_2, \tau_1))=1$ then $R_{\tau}/I_{\tau}$ and $R_{\tau}^{\rm red}/I_{\tau}^{\rm red}$ are cyclic $\Oo$-modules by Corollary 7.12 in \cite{BergerKlosin13} which combined with Proposition \ref{tor} implies finiteness of  $R_{\tau}/I_{\tau}$ and $R_{\tau}^{\rm red}/I_{\tau}^{\rm red}$. On the other hand given that $\dim_{\bfF}H^1_{\Sigma}(F, \Hom(\tau_2, \tau_1))>1$ it is easy to construct an upper-triangular (not necessarily crystalline) lift of $\tau$ to $\bfF[[X]]$ which would suggest that in general $R_{\tau}/I_{\tau}$, and even $R_{\tau}^{\rm red}/I_{\tau}^{\rm red}$ (since $\bfF[[X]]$ is reduced),  may have positive Krull dimension. Indeed, to see this, let $f$ be a cohomology class corresponding to $\tau$ and let $g$ be a cohomology class linearly independent from $f$. Then the representation $$\rho= \bmat \tau_1 & \tau_2(f + gX)\\ 0 & \tau_2\emat$$ is a non-trivial lift of $\tau$ to $\GL_n(\bfF[[X]])$. In particular there is no guarantee that $R_{\tau}^{\rm red}$ is a finitely generated $\Oo$-module. Since our method of proving modularity relies on that property we will restrict in the following section to the `characteristic zero' part of $R_{\tau}^{\rm red}$ of which we will demand that it is finite over $\Oo$. \end{rem}

\subsection{The ring $R_{\tau}^0$} \label{r0section}
Set $\mP(\tau):=\{\fp \in \Spec(R_{\tau}) \mid R_{\tau}/\fp = \Oo\}$. 
For the rest of this article we assume the following:
\begin{assumption} \label{finiteness}
Assume that $\mP(\tau)$ is finite.
\end{assumption} We then define $R_{\tau}^0$ to be the  image of $R_{\tau}$ in $\prod_{\fp \in \mP(\tau)}\Oo$.  
It is clear that $R^0_{\tau}$ is  a finitely generated $\Oo$-module  and  an object in $\textup{LCN}(E)$. Note that the canonical surjection $R_{\tau} \twoheadrightarrow R_{\tau}^0$ factors through $R_{\tau}^{\rm red}$. Write $\rho_{\tau}^0$ for the composition of $\rho_{\tau}$ with the map $\varphi_{\tau}: R_{\tau} \twoheadrightarrow R_{\tau}^0$. Write $I_{\tau}^0$ for the ideal of reducibility of $\tr \rho_{\tau}^0$. By \cite{BergerKlosin13}, Lemma 7.11, we have $\varphi_{\tau}(I_{\tau}) \subset I_{\tau}^0$ (in fact equality holds since the opposite inclusion is obvious) and thus $\varphi_{\tau}$ induces a surjection $R_{\tau}/I_{\tau} \twoheadrightarrow R_{\tau}^0/I_{\tau}^0$.

\begin{lemma} \label{finiteq} The quotient $R_{\tau}^0/I_{\tau}^0$ is finite. \end{lemma}
\begin{proof} This follows immediately from Proposition \ref{tor} and the surjectivity of $R_{\tau}/I_{\tau} \rightarrow R_{\tau}^0/I_{\tau}^0$.\end{proof}

Define $R_{\tau}^{\rm tr, 0} \subset R_{\tau}^0$ to be the closed $\Oo$-subalgebra generated by the set $$S:= \{\tr \rho_{\tau}^0(\Frob_{\fq}) \mid \fq \not\in \Sigma\}.$$

\begin{lemma} \label{Rtr0}  The image of  $R_{\tau}^{\rm tr}$ under $\varphi_{\tau}: R_{\tau} \twoheadrightarrow R_{\tau}^0$ is $R_{\tau}^{\rm tr, 0}$. Thus $R_{\tau}^{\rm tr, 0}$ is an object in the category $\textup{LCN(E)}$.  \end{lemma}

\begin{proof} It is clear that $R_{\tau}^{\rm tr, 0}\subset \varphi_{\tau}(R_{\tau}^{\rm tr})$. On the other hand $S \subset \varphi_{\tau}(R_{\tau}^{\rm tr})$, so the equality holds because $S$ is dense in $R_{\tau}^{\rm tr, 0}$. \end{proof}

We will write $I^{\rm tr, 0}_{\tau}\subset R_{\tau}^{\rm tr, 0}$ for the ideal of reducibility of $\tr \rho^0_{\tau}$. By Lemma \ref{Rtr0} and Lemma 7.11 in \cite{BergerKlosin13} we get that $\varphi_{\tau}(I^{\rm tr}_{\tau}) \subset I_{\tau}^{\rm tr, 0}$ (in fact equality holds)  and thus $\varphi_{\tau}$ induces a surjection $R_{\tau}^{\rm tr} /I_{\tau}^{\rm tr} \twoheadrightarrow R_{\tau}^{\rm tr, 0 } /I_{\tau}^{\rm tr, 0}$. By Remark \ref{r2.10} the quotient $R_{\tau}^{\rm tr, 0 } /I_{\tau}^{\rm tr, 0}$ is a cyclic $\Oo$-module.

\subsection{Generic irreducibility of $\rho_{\tau}^{0}$}

\begin{lemma} \label{irr4}
For any $\tau$ as in (\ref{form1})
 $\rho^{0}_{\tau}\otimes_{R^{0}_{\tau}}\mF$ is irreducible. Here $\mF$ is any of the fields $\mF_s$ in $\mF_{\tau}^{0}=\prod_s \mF_s$, where $\mF_{\tau}^{0}$ is the total ring of fractions of $R^{0}_{\tau}$.
\end{lemma}

\begin{proof} First note that since $R^{0}_{\tau}$ is a finitely generated $\Oo$-module and since $E$ is assumed to be sufficiently large we may assume that all of the fields $\mF_s$ are equal to $E$. If any of the representations $\rho^{0}_{\tau}\otimes_{R^{0}_{\tau}}\mF$ is reducible write $\rho=\bigoplus_{j=1}^s \rho_j$ for its semi-simplification with each $\rho_j$ irreducible, $j=1,2,\dots, s$. Then by compactness of $G_{\Sigma}$
for each $1\leq j \leq s$ there exists a $G_{\Sigma}$-stable $\Oo$-lattice inside the representation space of $\rho_j$. This implies that $\tr \rho_j(\sigma) \in \Oo$ for all $\sigma \in G_{\Sigma}$ and all $1 \leq j \leq s$.  Hence $\tr \rho$ splits over $\Oo$ into the sum of traces of $\rho_j$. Since $\rho^{0}_{\tau}$ is a deformation of $\tau$ we easily conclude that $\rho=\rho_1 \oplus \rho_2$ with $\rho_j$ (with respect to some lattice) being a deformation of $\tau_j$ ($j=1,2$). Using the fact that $\rho^{0}_{\tau}$ is a deformation of $\tau$ we now deduce that there is an $\Oo$-lattice inside the space of $\rho^{0}_{\tau}\otimes_{R^{0}_{\tau}}\mF$ with respect to which $\rho^{0}_{\tau}\otimes_{R^{0}_{\tau}}\mF$ is block-upper-triangular (with correct dimensions) and non-semi-simple. When we reduce it modulo $\varpi^m$, the upper-right shoulder will give rise to an element of order $\varpi^m$ in $ H^1_{\Sigma}(F, \Hom_{\Oo}(\tilde{\tau}_2, \tilde{\tau}_1)\otimes_{\Oo}E/\Oo)$. Since $m$ is arbitrary this contradicts Assumption \ref{BK}. \end{proof}

\section{The rings $\bfT_{\tau}$} \label{The rings Ttau}
Let us now define the rings $\bfT_{\tau}$ that will correspond to 
$R^{0}_{\tau}$ on the Hecke side.

\begin{prop} \label{genRibet} If $\rho: G_{\Sigma} \rightarrow \GL_n(E)$ is irreducible and satisfies  \be \label{semis} \ov{\rho}^{\rm ss} \cong \tau_1 \oplus \tau_2\ee then there exists a lattice  inside $E^{n}$ so that with respect to that lattice
the mod $\varpi$ reduction $\ov{\rho}$ of $\rho$ has the form $$\ov{\rho}=\bmat \tau_1 & * \\ 0 & \tau_2\emat$$ and is
non-semi-simple. \end{prop}

\begin{proof} This is a special case of \cite{Urban01}, Theorem 1.1, where the ring $\mB$ in [loc.cit.] is a discrete valuation ring $=\Oo$.
\end{proof}

For each representation $\tau$ as in (\ref{form1}) let $\Phi_{\tau}$ be the set of (inequivalent) 
 characteristic zero deformations of $\tau$, i.e. crystalline at $\fp\mid p$ Galois representations 
$\rho: G_{\Sigma} \rightarrow \GL_n(\Oo)$ whose reduction equals $\tau$. 
Also, let $\Phi_{\tau, E}$ be the set of (inequivalent) crystalline at $\fp \mid p$ Galois representations $\rho: G_{\Sigma} \to \GL_n(E)$ such that there exists a $G_{\Sigma}$-stable lattice $L$ in the space of $\rho$ so that the mod $\varpi$-reduction of $\rho_L$ equals $\tau$.

The following is a higher-dimensional analogue of Lemma 2.13(ii) from \cite{SkinnerWiles99}:
\begin{prop} \label{SW2.13} One has $\Phi_{\tau,E} \cap \Phi_{\tau',E}=\emptyset$ if $\tau \not\cong \tau'$. \end{prop}
\begin{proof} 
Let $\rho:G_{\Sigma} \to {\rm GL}_n(E)$ be a representation such that $\ov{\rho}^{\rm ss}=\tau_1 \oplus \tau_2$ and let $T$ equal its trace.  Suppose there exist two lattices  $L_i$ in the representation space of $\rho$ such that the reductions of the corresponding representations $\rho_{L_i}$ are given by $\tau$ and $\tau'$ with $\tau \not\cong \tau'$ as in  (\ref{form1}). We now consider the classes $c_{L_i}$ of the cocycles corresponding to $\ov{\rho}_{L_i}$ in  ${\rm Ext}^1_{\bfF[G_{\Sigma}]/\ker T}({\tau_2},{\tau_1})$. 
Using Assumption \ref{BK} above and Corollary 7.8 in \cite{BergerKlosin13} we conclude that the quotient $\Oo/I_T$ is finite. 
Thus arguing as in the proof of Proposition 1.7.4 in \cite{BellaicheChenevierbook} but using Proposition 3.1 in \cite{BergerKlosin13} instead of generic irreducibility of $T$ to conclude that $\ker T=\ker \rho$ (see \cite{BellaicheChenevierbook}, Proof of Proposition 1.7.2, on how this equality - which follows from Proposition 1.6.4 in [loc.cit.] in the generically irreducible case - is used) 
we obtain that the existence of $\rho_{L_i}$ with trace $T$ and non-split reduction as in (\ref{form1}) implies that ${\rm Ext}^1_{X}({\tau_2},{\tau_1})$ is 1-dimensional, where $X:= (\Oo[G_{\Sigma}]/\ker T)/\varpi(\Oo[G_{\Sigma}]/\ker T)$.

 First note that $X = \Oo[G_{\Sigma}]/(\varpi \Oo[G_{\Sigma}] + \ker T)$. Secondly one clearly has that $\ker (\Oo[G_{\Sigma}] \twoheadrightarrow \bfF[G_{\Sigma}]) = \varpi \Oo[G_{\Sigma}]$. These two facts imply that the map $\Oo[G_{\Sigma}] \twoheadrightarrow X$ factors through $\Oo[G_{\Sigma}] \twoheadrightarrow \bfF[G_{\Sigma}]$ and that the kernel of the resulting surjection $\bfF[G_{\Sigma}] \twoheadrightarrow X$ equals $(\ker T)\bfF[G_{\Sigma}]$. Thus we have $X=\bfF[G_{\Sigma}]/(\ker T)\bfF[G_{\Sigma}]$, so by the above we conclude that 
${\rm Ext}^1_{\bfF[G_{\Sigma}]/\ker T}({\tau_2},{\tau_1})$ is one-dimensional. 
This means the corresponding representations of $\bfF[G_{\Sigma}]/\ker T$ are isomorphic. Since $\ker T =\ker \rho$ (as noted above)
the reductions both factor through this quotient of $\bfF[G_{\Sigma}]$,  and so they are isomorphic as representation of $\bfF[G_{\Sigma}]$, in contradiction to our assumption.
\end{proof}

The following notation will remain in force throughout the paper.
\begin{notation} \label{notation1} Write $\fT$ for the set of isomorphism classes of residual representations of the form (\ref{form1}). Set $\Phi = \bigcup_{\tau \in \fT} \Phi_{\tau}$. 
\end{notation}
\begin{rem} \label{Phirem} Assumption \ref{finiteness} that $\mP(\tau)$ is finite 
 is equivalent to assuming that the set $\Phi_{\tau}$ is a finite set. 
 \end{rem}

We now fix subsets $\Pi_{\tau} \subset \Phi_{\tau}$ and $\Pi \subset \Phi$ of deformations. In our later application these will be taken to correspond to all the modular deformations corresponding to cuspforms of a particular weight and level which are congruent to a fixed Eisenstein series. In particular $\Pi_{\tau}$ may be empty.

Whenever $\Pi_{\tau} \neq \emptyset$ we obtain an $\Oo$-algebra map $ R_{\tau} \rightarrow \prod_{\rho \in \Pi_{\tau}} \Oo.$  This induces a map  \be \label{defhec} R^{\rm tr}_{\tau} \rightarrow \prod_{\rho \in \Pi_{\tau}} \Oo.\ee
\begin{definition} \label{def hecke} We (suggestively) write
$\bfT_{\tau}$ for the image of the map (\ref{defhec}) - note that this also depends on the choice of the set $\Pi_{\tau}$ - and denote the resulting surjective $\Oo$-algebra map $R^{\rm tr}_{\tau} \twoheadrightarrow \bfT_{\tau}$ by
$\phi_{\tau}$. Also we will write $\bfT$ for the image of $\phi: R^{\rm ps} \to \prod_{\rho \in \Pi} \Oo$, where $\phi$ is induced from the traces of the  deformations $\rho_{\pi}$. Finally we will write $J_{\tau}\subset \bfT_{\tau}$ for the ideal of reducibility of the pseudo-representation $\bfT_{\tau} \otimes_{R^{\rm tr}_{\tau},\phi_{\tau}}  \tr \rho_{\tau}: \bfT_{\tau}[G_{\Sigma}] \to \bfT_{\tau}$ and $J \subset \bfT$ for the ideal of reducibility of the pseudo-representation $T^{\rm ps} \otimes_{R^{\rm ps},\phi} \bfT: \bfT[G_{\Sigma}] \to \bfT$.
\end{definition}

\begin{lemma}\label{factoring}  The maps $R_{\tau}^{\rm tr} \twoheadrightarrow \bfT_{\tau}$ and $R^{\rm ps} \twoheadrightarrow \bfT$ factor through $R^{\rm tr, 0}_{\tau}$ and the image of $R^{\rm ps}$ inside $R^{\rm ps} \otimes_{\Oo} E$ respectively. \end{lemma}
\begin{proof}  Clearly the kernel of $R_{\tau} \rightarrow \prod_{\rho \in \Pi_{\tau}} \Oo$ contains $\bigcap_{\fp \in \mP(\tau)} \fp$. Thus the map $R_{\tau} \rightarrow \prod_{\rho \in \Pi_{\tau}} \Oo$ factors through $R^0_{\tau}$. Then the claim follows since $\varphi_{\tau}(R_{\tau}^{\rm tr }) = R_{\tau}^{\rm tr, 0}$ by Lemma \ref{Rtr0}.  \end{proof}

\begin{lemma}  \label{cyclicity} The quotient $\bfT_{\tau}/J_{\tau}$ is cyclic and one has $J_{\tau}=\phi_{\tau}(I^{\rm tr}_{\tau})$. \end{lemma}

\begin{proof} 
The first part is a consequence of Lemma \ref{l2.6} and was already mentioned in Remark \ref{r2.10}. 

By Lemma 7.11 in \cite{BergerKlosin13} we know that  $J_{\tau}\supset \phi_{\tau}(I^{\rm tr}_{\tau})$. For the opposite inclusion we argue as follows. We need to show that  $\phi_{\tau} \circ \tr \rho_{\tau} \equiv \Psi'_1 + \Psi'_2 \mod{\phi_{\tau}(I^{\rm tr}_{\tau})}$ for $\Psi'_1, \Psi'_2$ pseudo-representations.

Put $B=\bfT$, $A=R^{\rm tr}_{\tau}$ and write $\varphi$ for $\phi_{\tau}:R^{\rm tr}_{\tau} \to \bfT$ and $T_B$ for $\bfT_{\tau} \otimes_{R^{\rm tr}_{\tau}, \phi_{\tau}} \tr \rho_{\tau}$.
Let $x\in B[G_{\Sigma}]$. Since $\varphi$ is surjective there exists $y\in A[G_{\Sigma}]$ such that $\varphi(y)=x$. Then by definition of $T_B$ we have $T_B(x) = \varphi \circ T(y) = \varphi (\Psi_1(y) + \Psi_2(y) +i)$ for some pseudo-representations $\Psi_1, \Psi_2$ and $i \in I^{\rm tr}_{\tau}$.  Now set $\Psi'_j(x):= \varphi\circ\Psi_j(y)$ for $j=1,2$.
 \end{proof}

\begin{cor} \label{cyclicity2} One has $J_{\tau}=\phi_{\tau}(I^{\rm tr, 0}_{\tau})$. \end{cor}
\begin{proof} By Lemma \ref{factoring} the map $\phi_{\tau}$ factors through $R_{\tau}^{\rm tr, 0}$. By abuse of notation we will denote the induced map also by $\phi_{\tau}$ as in the statement of the Corollary. Then since $\varphi_{\tau}(I_{\tau}^{\rm tr}) = I_{\tau}^{\rm tr, 0}$ (with $\varphi_{\tau}$ as in section  \ref{r0section}) we get the corollary.  \end{proof}

\begin{lemma} \label{Tcyclic}
 The quotient $\bfT/J$ is cyclic.
\end{lemma}

\begin{proof}
For this we prove as in Lemma \ref{l2.6} that $J$ is equal to the smallest closed ideal of $T$ generated by the set $\{  (\phi \circ T^{\rm ps})(\Frob_v) - \tr \tilde{\tau}_1(\Frob_v) - \tr \tilde{\tau}_2(\Frob_v) \mid v \not\in \Sigma \}.$ We note that Assumption \ref{ass1} can again be applied as Definition  \ref{def hecke} tells us that $\phi$ is induced by the traces of the crystalline deformations $\rho_{\pi}$.
\end{proof}

\section{The lattice $\mL$ and modular extensions}\label{lattice}

We will make a frequent use of the following result that is due to Urban \cite{Urban01}. Let $\mB$ be a Henselian and reduced local commutative algebra that is a finitely generated $\Oo$-module. Since $\Oo$ is assumed to be sufficiently large and $\mB$ is reduced we have $$\mB \subset \hat{\mB} = \prod_{i=1}^s \Oo \subset \prod_{i=1}^s E = \mF_{\mB},$$ where $\hat{\mB}$ stands for the normalization of $\mB$ and $\mF_{\mB}$ for its total ring of fractions. Write $\fm_{\mB}$ for the maximal ideal of $\mB$. For any finitely generated free $\mF_{\mB}$-module $M$, any $\mB$-submodule $N \subset M$ which is finitely generated as a $\mB$-module and has the property that $N \otimes_{\mB} \mF_{\mB} = M$ will be called a \emph{$\mB$-lattice}. 

\begin{thm}[\cite{Urban01} Theorem 1.1] \label{Urban 1.1}
Let $\mR$ be a $\mB$-algebra, and let $\rho$ be an absolutely irreducible representation of $\mR$ 
on $\mF_{\mB}^n$ (i.e., $\rho$ composed with each of the projections $\mF_{\mB}\twoheadrightarrow E$ is absolutely irreducible) such that there exist two representations $\rho_i$ for $i = 1, 2$ in $M_{n_i} (\mB)$ and $I$
a proper ideal of $\mB$ such that \begin{itemize}
\item[(i)] the coefficients of the characteristic polynomial of $\rho$ belong to $\mB$;
\item[(ii)] the characteristic polynomials of $\rho$ and $\rho_1 \oplus \rho_2$ are congruent modulo $I$;
\item[(iii)] $\ov{\rho}_1:=\rho_1$ mod $\fm_{\mB}$ and $\ov{\rho}_2:=\rho_2$ mod $\fm_{\mB}$ are absolutely irreducible;
\item[(iv)] $\ov{\rho}_1 \not\cong \ov{\rho}_2$. \end{itemize}
Then there exist an $\mR$-stable $\mB$-lattice $\mL$ in $\mF_{\mB}^n$ 
and a $\mB$-lattice $\mT$ of $\mF_{\mB}$ such that
we have the following exact sequence of $\mR$-modules:
$$0 \to \rho_1 \otimes_{\mB} \mT/I \mT \to \mL \otimes_{\mB} \mB/I \to \rho_2 \otimes_{\mB} \mB/I \to 0$$ which splits as a sequence of $\mB$-modules.
 Moreover, $\mL$ has no quotient isomorphic
to $\ov{\rho}_1$. \end{thm}

Since we will not only use Theorem \ref{Urban 1.1} itself but also the construction of the lattice $\mL$ let us briefly summarize how $\mL$ is built (for details cf. [loc.cit.], p. 490-491). Let $\rho_i$ be the composition of the representation $\rho$ with the projection $\mB \twoheadrightarrow \Oo$ onto the $i$th component of $\hat{\mB}$. Urban shows that we can always conjugate $\rho_i$ (over $E$) so that the mod $\varpi$ -reduction of (the conjugate of $\rho_i$ which we will from now on denote by) $\rho_i$ has the form \be \label{form} \bmat \ov{\rho}_1 & * \\  0 &\ov{\rho}_2\emat.\ee Set $\rho_{\hat{\mB}}:= (\rho_i)_i$. It is also shown in [loc.cit.] that
the matrices $\bmat I_{n_1} &0 \\ 0&0 \emat$ and $\bmat 0&0\\ 0&I_{n_2} \emat$
are in the image of $\rho_{\hat{\mB}}$. One then defines the lattice $\mL$ to be the $\mB$-submodule of $\hat{\mB}^n$ generated by $\rho(r) {}^t\bmat 0,0, \dots, 0,1\emat$, where $r$ runs over $\mR$ and set $\mL^1:= \bmat I_{n_1} &0 \\ 0&0 \emat \mL$ and $\mL^2:= \bmat 0&0\\ 0&I_{n_2} \emat \mL$.

Let $\fT$ be as in Notation \ref{notation1}. Let $\bfT$, $\Pi \subset \Phi$ and $\Pi_{\tau} \subset \Phi_{\tau}$ be as in section \ref{The rings Ttau}.
Write $\tilde{\bfT}=\prod_{\pi \in \Pi} \Oo$ for the normalization of $\bfT$. Let $\rho$ in Theorem  \ref{Urban 1.1} be $\rho_{\Pi}=\prod_{\tau \in \fT} \prod_{\rho_{\pi} \in \Pi_{\tau}} \rho_{\pi}=\prod_{\pi \in \Pi} \rho_{\pi}$ and $\rho_i=\tilde{\tau}_i$, $i=1,2$,
where $\tilde{\tau}_i: G_{\Sigma} \to \GL_{n_i}(\Oo)$ is a fixed crystalline deformation of $\tau_i$ which we from now on assume exists. (If one works under Assumption \ref{ass1}, then the $\tilde{\tau}_i$'s are unique, but we do not need this uniqueness for the arguments of this section.)
 Note that the reduction of $\rho_{\pi}$ already has the form (\ref{form}), so
we can take $\rho_{\hat{\mB}} = \rho_{\Pi}$ and define lattices $\mL$, $\mL^1$ and $\mL^2$ as above  (with $\mB=\bfT$, $\mR=\bfT[G_{\Sigma}]$).
The $G_{\Sigma}$-action on $\mL$ is then via restriction of $\rho_{\Pi}$ to $\mL$.
Write $\fm=\fm_{\bfT}$ for the maximal ideal of the local ring $\bfT$ and let $J$ as in section \ref{The rings Ttau} be its reducibility ideal.

By Theorem \ref{Urban 1.1} (and Lemmas 1.1 and 1.5(ii) in \cite{Urban01}) there exists a $\bfT$-lattice $\mT$ and a short exact sequence of $\bfT[G_{\Sigma}]$-modules (which splits as a sequence of $\bfT$-modules):
\be \label{Jseq} 0 \to \mL^1\otimes_{\bfT} \bfT/J \to \mL\otimes_{\bfT} \bfT/J\to \mL^2\otimes_{\bfT} \bfT/J \to 0 \ee with
 $$\mL^1\otimes_{\bfT} \bfT/J \cong  \tilde{\tau}_1\otimes_{\Oo}\mT/J\mT \quad \textup{and} \quad \mL^2 \otimes_{\bfT}\bfT/J \cong \tilde{\tau}_2 \otimes_{\Oo}\bfT/J$$ as $\bfT$-modules where the $\bfT$-action on $ \tilde{\tau}_1\otimes_{\Oo}\mT/J\mT$ is via the second factor.

Note that we have the following identification
\be \label{iden2}\Hom_{\Oo}(\tilde{\tau}_2, \tilde{\tau}_1)\otimes_{\Oo}\mT/J\mT \xrightarrow{\sim} \Hom_{\bfT/J}(\mL^2\otimes_{\bfT} \bfT/J, \mL^1\otimes_{\bfT} \bfT/J).\ee
Let $s: \mL^2\otimes_{\bfT} \bfT/J \to \mL\otimes_{\bfT} \bfT/J$ be a section of $\bfT/J$-modules of (\ref{Jseq}). 
Using (\ref{iden2}) as in \cite{Klosin09}, p.159-160, we define a cohomology class $c \in H^1(F_{\Sigma}, \Hom_{\Oo}(\tilde{\tau}_2, \tilde{\tau}_1)\otimes_{\Oo}\mT/J\mT)$ 
by 
$$g \mapsto (\lambda_2 \otimes t \mapsto s(\lambda_2\otimes t) - g \cdot s(g^{-1} \cdot \lambda_2\otimes t)).$$
We also define a map
$$\iota_J: \Hom_{\Oo}(\mT/J\mT, E/\Oo) \to H^1(F_{\Sigma}, \Hom_{\Oo}(\tilde{\tau}_2, \tilde{\tau}_1)\otimes_{\Oo}E/\Oo), \quad f \mapsto (1 \otimes f)(c).$$
 Let us just briefly remark that $\iota_J$ is independent of the choice of the section $s$. 
From now on we will make the following assumption on the quotient $\bfT/J$. 
\begin{assumption} \label{bound on T/J} One has $$\#\bfT/J \geq \#\Oo/L$$ with $L$ as in Assumption \ref{BK}. \end{assumption}
\begin{rem} In Section \ref{section4} we will describe a particular setup for $n=2$ and $F$ an imaginary quadratic field under which  Assumptions \ref{BK} and \ref{bound on T/J} are satisfied. However, we expect that these conditions hold also for other CM fields (for $n=2$), and have therefore presented the results of this and the following sections under these two general assumptions. \end{rem}

\begin{lemma} \label{iota43J} If Assumptions \ref{BK} and \ref{bound on T/J} hold, then the map $$\iota_J: \Hom_{\Oo}(\mT/J\mT, E/\Oo) \to H^1(F_{\Sigma},\Hom_{\Oo}(\tilde{\tau}_2, \tilde{\tau}_1)\otimes_{\Oo}E/\Oo )$$ is injective and its image equals $H^1_{\Sigma}(F, \Hom_{\Oo}(\tilde{\tau}_2, \tilde{\tau}_1)\otimes_{\Oo}E/\Oo))$. \end{lemma}

\begin{proof} For the injectivity of $\iota_J$ and for the fact that its image lands in the Selmer group one follows the strategy in \cite{Berger05}, p.119-120 which was later spelled out in a higher dimensional case in \cite{Klosin09}, Lemmas 9.25 and 9.26. Let us outline the argument here. Let $f \in \ker \iota_J$ and set $K_f:=(\mT/J\mT)/\ker f$, $I_f:=(E/\Oo)/\textup{Im} f$ and $\tilde{T}:=  \Hom_{\Oo}(\tilde{\tau}_2, \tilde{\tau}_1)$. Tensoring the exact sequence $0 \to K_f \xrightarrow{f} E/\Oo \to I_f \to 0$ with $\otimes_{\Oo}\tilde{T}$ we get the exactness of the bottom row of the following commutative diagram:
\be\label{9.13} \xymatrix{&H^1(G_{\Sigma}, \tilde{T}\otimes_{\Oo}\mT/J\mT) \ar[d]^{\phi}\ar[dr]^{H^1(1\otimes f)}\\
H^0(G_{\Sigma}, \tilde{T}\otimes_{\Oo}I_f) \ar[r] & H^1(G_{\Sigma}, \tilde{T}\otimes_{\Oo} K_f) \ar[r]^{H^1(1 \otimes f)} & H^1(G_{\Sigma}, \tilde{T}\otimes_{\Oo}E/\Oo).}\ee Clearly $H^1(1\otimes f)\circ\phi(c)=0$ and since the first term in the bottom row vanishes (as a consequence of absolute irreducibility of $\tilde{\tau}_i$ and the fact that $\tilde{\tau}_1 \not\cong \tilde{\tau}_2$) we get $\phi(c)=0$. Assuming $f \neq 0$, one constructs an $\Oo$-module $A\subset \mT/J\mT$ containing $\ker f$ such that $(\mT/J\mT)/A=\Oo/\varpi$. It is easy to show that $\phi(c)=0$ implies the splitting of the following exact sequence of $\bfT[G_{\Sigma}]$-modules
$$0 \to \tau_1 \to (\mL/J\mL)/(\varpi \mL + \tilde{\tau}_1\otimes_{\Oo} A) \to \tau_2 \to 0$$ contradicting the fact that $\mL$ has no quotient isomorphic to $\tau_1$ (cf. Theorem \ref{Urban 1.1}). This proves injectivity of $\iota_J$. 

On the other hand the fact that $\textup{Im} (\iota_J)$ is contained in the Selmer group can be deduced from the fact that each of the representations $\rho_{\pi}$ making up $\rho = \rho_{\Pi} = \prod_{\pi \in \Pi} \rho_{\pi}$ is crystalline because it implies that the cohomology class $c$ is also crystalline (see \cite{Klosin09}, proof of Lemma 9.25 for more details).

Using \cite{Klosin09}, Lemma 9.21 (which is just a slightly expanded version of Theorem \ref{Urban 1.1}) we get $\textup{Fitt}_{\bfT} \mT = 0$.

By Lemma \ref{Tcyclic} we know that $\bfT/J=\Oo/\varpi^n$ for some $n$. Recall property 4 from the Appendix in \cite{MazurWiles84}:
For an $R$-module $M$ and an ideal $I \subset R$ we have \be\label{MW} \Fit_{R/I} (M/IM) = \Fit_R(M) + I \subset R/I. \ee
Since  $\Fit_{\bfT}(\mT)=0$ this implies that $\Fit_{\Oo/\varpi^n}(\mT/J\mT)=\Fit_{\bfT/J}(\mT/J\mT)=(0)$ in $\bfT/J=\Oo/\varpi^n$.

Note that $\varpi^n$ annihilates $\mT/J\mT$, so using (\ref{MW}) again we get $\Fit_{\Oo/\varpi^n}(\mT/J\mT)= \Fit_{\Oo/\varpi^n}((\mT/J\mT)/\varpi^n) =\Fit_{\Oo}(\mT/J\mT)+\varpi^n\Oo$.

Together this shows that $\Fit_{\Oo}(\mT/J\mT)$ maps to the $0$-ideal in  $\Oo/\varpi^n$, i.e.  $$\Fit_{\Oo}(\mT/J\mT) \subset \varpi^n \Oo=\Fit_{\Oo}(\bfT/J).$$

By property 11 in the Appendix of \cite{MazurWiles84} we know that $\varpi^{{\rm length}_{\Oo}(\mT/J\mT)}\Oo \subset \Fit_{\Oo}(\mT/J\mT)$.
This, combined with Assumption \ref{bound on T/J}  
and Assumption \ref{BK},  implies that $\iota_J$ must in fact surject onto the Selmer group. \end{proof}

Since (\ref{Jseq}) splits as a sequence of $\bfT/J$-modules we can tensor it with $\otimes_{\bfT/J}\bfF$ and obtain an exact sequence of $\bfF[G_{\Sigma}]$-modules 
$$0 \to \mL^1\otimes_{\bfT} \bfF \to \mL\otimes_{\bfT} \bfF\to \mL^2 \otimes_{\bfT} \bfF \to 0 $$ with $$\mL^1\otimes_{\bfT} \bfF  \cong  \mT\otimes_{\bfT} \tau_1 \quad \textup{and} \quad \mL^2 \otimes_{\bfT} \bfF \cong \tau_2.$$
Arguing as above (with $\fm_{\bfT}$ instead of $J$) we again obtain an injective map $$\iota: \Hom_{\Oo}(\mT/J \mT, \bfF) \to H^1(F_{\Sigma},\Hom_{\bfF}(\tau_2, \tau_1)).$$

\begin{lemma} \label{iota43} Suppose that Assumptions \ref{BK} and \ref{bound on T/J} hold. The map $\iota: \Hom_{\Oo}(\mT/J\mT, \bfF) \to H^1(F_{\Sigma},\Hom_{\bfF}(\tau_2, \tau_1))$ is injective and its image equals  $H^1_{\Sigma}(F,\Hom_{\bfF}(\tau_2, \tau_1))$. \end{lemma}
\begin{proof} We have the following commutative diagram $$\xymatrix{\Hom_{\Oo}(\mT/J\mT, E/\Oo) \ar[r]^-{\iota_J}  &  H^1(F_{\Sigma},\Hom_{\Oo}(\tilde{\tau}_2, \tilde{\tau}_1)\otimes_{\Oo}E/\Oo) \\
\Hom_{\Oo}(\mT/J \mT, \bfF)\ar[u]\ar[r]^{\iota} &  H^1(F_{\Sigma},\Hom_{\bfF}(\tau_2, \tau_1))\ar[u]}$$
Denote the right vertical arrow by $f$. Lemma \ref{iota43J} implies that the image of $f \circ \iota$ is contained in the $\varpi$-torsion of $H^1_{\Sigma}(F,\Hom_{\Oo}(\tilde{\tau}_2, \tilde{\tau}_1)\otimes_{\Oo}E/\Oo)$. 
 Moreover, by Proposition 5.8 in \cite{BergerKlosin13}, we know that the $\varpi$-torsion in $H^1_{\Sigma}(F, \Hom_{\Oo}(\tilde{\tau}_2, \tilde{\tau}_1)\otimes_{\Oo}E/\Oo)$ coincides with $f(H^1_{\Sigma}(F,\Hom_{\bfF}(\tau_2, \tau_1) ))$.
Since $f$ is injective,
 this implies that the image of $\iota$ is contained in the Selmer group. 
 Hence it remains to show that $\iota_J$ is an isomorphism on $\varpi$-torsion. But this is clear since $\iota_J$ is an isomorphism by Lemma \ref{iota43J}. \end{proof}

\begin{lemma} \label{gene} Suppose that Assumptions \ref{BK} and \ref{bound on T/J} hold. Write $\ov{\rho}_{\Pi}$ for $\prod_{\rho_{\pi} \in \Pi}\ov{\rho}_{\pi}$. The $\bfF[G_{\Sigma}]$-module $\mL \otimes_{\bfT} \bfF$ coincides with the $\bfF$-subspace of $\prod_{\rho_{\pi}\in \Pi} \bfF^n$ generated by $\ov{\rho}_{\Pi}(r) e_n$, where $r$ runs over $\bfF[G_{\Sigma}]$,  $e_n$ is a column matrix in $\bfF^n$ whose last entry is 1 and all the other ones are zero. \end{lemma}

\begin{proof}
By definition of $\mL$, every element of $\mL\otimes_{\bfT} \bfF$ can be written as $\sum_i t_i \rho_{\Pi}(g_i)e_n \otimes a_i$ with $t_i \in \bfT$, $a_i \in \bfF$ and $g_i \in G_{\Sigma}$. 
Writing $\ov{t}_i$ for the image of $t_i$ under the canonical map $\bfT \twoheadrightarrow \bfF$ we can re-write the above sum as $\sum_i \rho_{\Pi}(g_i)e_n \otimes a_i\ov{t}_i$ and $a_i\ov{t}_i \in \bfF$. It suffices to show now that for every $g\in G_{\Sigma}$ we get $\rho_{\Pi}(g) e_n \otimes 1 = \ov{\rho}_{\Pi}(g)e_n\otimes 1.$
Write $$\rho_{\Pi}(g) = \bmat a_{11}(g) + a'_{11}(g) & a_{12}(g) + a'_{12}(g)\\
a'_{21}(g) & a_{22}(g) + a'_{22}(g)\emat,$$ where $a_{11}, a'_{11}$ are $(n-1)\times (n-1)$-matrices, $a_{22}, a'_{22}$ are scalars and the other matrices have sizes determined by these two and the entries of $a'_{ij}(g)$ lie in  $\varpi \Oo \oplus \varpi \Oo \oplus \dots \oplus \varpi\Oo$ 
Thus, $$\rho_{\Pi}(g) e_n \otimes 1 =  \bmat a_{12}(g) + a'_{12}(g) \\ a_{22}(g) + a'_{22}(g)\emat \otimes 1 =  \bmat a_{12}(g)  \\ a_{22}(g) \emat \otimes 1 + \bmat a'_{12}(g)/\varpi  \\ a'_{22}(g)/\varpi \emat \otimes \varpi,$$ and the latter tensor is zero. This proves the lemma.
\end{proof}

Let us now turn to the 2-dimensional situation, where every $\tau$ is 
 (up to a twist) of the form $$\tau = \bmat 1 & * \\ 0 & \chi \emat$$ for a Galois character $\chi$. Note that $\Hom_{\Oo}(\mT/J\mT, \bfF) = \Hom_{\Oo}(\mT\otimes_{\bfT}\bfT/J, \bfF) =\Hom_{\Oo}(\mT\otimes_{\bfT}\bfT/\fm, \bfF)= \Hom_{\Oo}(\mT\otimes_{\bfT} \bfF, \bfF).$
\begin{prop} \label{modext1} Suppose that Assumptions \ref{BK} and \ref{bound on T/J} hold. The image of $\iota : \Hom(\mT\otimes_{\bfT} \bfF, \bfF) \hookrightarrow H^1_{\Sigma}(F, \chi^{-1})$ is spanned by extensions $\tau$ such that $\Pi_{\tau} \neq \emptyset$. \end{prop} 

\begin{proof} Let $\Pi$ be as above and $\ov{\Pi}$ be a subset $\Pi$ consisting of representatives of distinct isomorphism classes of residual representations (i.e., one element from every non-empty $\Pi_{\tau}$).
 By Lemma \ref{gene} the lattice $\mL \otimes_{\bfT} \bfF$ is generated by vectors $$x=\bmat (1,1, \dots, 1) & \chi(g)\alpha(g)\\ (0,0, \dots, 0)
& (\chi(g), \chi(g), \dots, \chi(g))\emat \bmat 0 \\ 1 \emat.$$ Let us explain the notation: There are $r$ elements in $\ov{\Pi}$ (which we will denote by $\tau^1, \tau^2, \dots, \tau^r$), $\sigma:=\sum_{i=1}^r s_i$ elements in $\Pi$. Moreover, $\alpha$ is a $\sigma$-tuple of functions such that  $\alpha(g)$ equals  $$(\alpha_{1,1}f_1(g), \dots, \alpha_{1,s_1}f_1(g),\alpha_{2,1} f_2(g), \dots, \alpha_{2,s_2}f_2(g), \dots, \alpha_{r,1}f_r(g), \dots \alpha_{r,s_r}f_r(g))\in \bfF^\sigma,$$ where the $\alpha_{i,j}$ are elements of $\bfF^{\times}$.
We get that $x$ equals $$ \chi(g) \bmat (\alpha_{1,1}f_1(g), \dots, \alpha_{1,s_1}f_1(g), \alpha_{2,1}f_2(g), \dots,\alpha_{2,s_2} f_2(g), \dots, \alpha_{r,1}f_r(g), \dots\alpha_{r,s_r} f_r(g)) \\ (1, 1, \dots, 1)\emat.$$ Let $\alpha^j$ be the $j$th entry of $\alpha$. Then we conclude that  $\mL\otimes_{\bfT} \bfF \cong \bfT \otimes_{\bfT} V = V$, where $V$ is the 
subspace of $(\bfF \oplus \bfF)^{\# \Pi}$ spanned over $\bfF$ by the set 
vectors of the form 
$$\left(\bmat \chi(g)\alpha^1(g)\\ \chi(g) \emat , \bmat \chi(g)\alpha^{2}(g)\\ \chi(g) \emat, \dots, \bmat \chi(g)\alpha^{\sigma}(g)\\ \chi(g) \emat\right).$$ For $j \in \{1,2, \dots, \sigma\}$ define integers $n(j)\in \{1,2, \dots, r\} $ and $m(j)\in \{1, 2, \dots, s_{n(j)}\}$ by the equality $$\alpha^j(g) = \alpha_{n(j), m(j)} f_{n(j)}(g).$$ The $G_{\Sigma}$-action on $V$ is via $\ov{\rho}_{\Pi}$, hence $h \in G_{\Sigma}$ acts on $\bmat \chi(g)\alpha^j(g)\\ \chi(g) \emat$ via the $n(j)$th residual representation in $\ov{\Pi}$, i.e., by multiplication by $\tau^{n(j)}(h)$. In particular all the vectors in $V$ have the form \be \label{form0} v=\left(\bmat a_1\\ a \emat , \bmat a_2\\ a \emat, \dots, \bmat a_{\sigma}\\ a \emat\right).\ee

By definition we have \be \label{form2} \mL^2\otimes_{\bfT} \bfF = \bmat 0&0 \\ 0&1 \emat \mL\otimes_{\bfT} \bfF = \bmat 0&0 \\ 0&1 \emat V \cong \bfF(\chi)\ee as $\bfF[G_{\Sigma}]$-modules, where we write $\bfF(\chi)$ for the one-dimensional $\bfF$-vector space on which $G_{\Sigma}$ acts via $\chi$. The surjective $\bfF[G_{\Sigma}]$-module map $V \twoheadrightarrow \bfF(\chi)$ is given by sending a vector $v$ as in (\ref{form0}) to $a$. Write $V'$ for the kernel of this map. Identifying $V'$ with $\mL^1 \otimes_{\bfT} \bfF = \bmat 1&0 \\ 0&0 \emat \mL\otimes_{\bfT} \bfF = \bmat 1&0 \\ 0&0 \emat V$ provides us with a splitting (only as $\bfF$-vector spaces) of the short exact sequence  of $\bfF[G_{\Sigma}]$-modules $$0 \to V' \to V \to \bfF(\chi) \to 0.$$ Since the $G_{\Sigma}$-action on $\mL^1\otimes_{\bfT} \bfF$ is trivial, we have $V'=\mL^1\otimes_{\bfT} \bfF = \mT \otimes_{\bfT} \bfF$. Clearly, we may assume that the vectors in $V'$ all have the form \be \label{form11} v_0=\left(\bmat a_1\\ 0 \emat , \bmat a_2\\ 0 \emat, \dots, \bmat a_{\sigma}\\ 0 \emat\right).\ee Let $\phi_j \in \Hom_{\bfF} (V', \bfF)=\Hom(\mT\otimes_{\bfT} \bfF, \bfF)$ be the homomorphism sending $v_0$ as in (\ref{form11}) to $a_j$.

Then the map $\iota$ sends $\phi_j$ to the cocycle $\alpha_{n(j), m(j)}f_{n(j)}$, i.e, to the residual representation $\bmat 1 & \chi\alpha_{n(j), m(j)}f_{n(j)} \\ 0 & \chi \emat$, which is isomorphic to the residual representation of the $n(j)$th element of $\ov{\Pi}$.
 So, this proves that the image of $\iota$ is spanned by modular extensions.  \end{proof}

\begin{cor} \label{surj1prime} If Assumptions \ref{BK} and \ref{bound on T/J} are satisfied then the space $H^1_{\Sigma}(F, \chi^{-1})$ has a basis consisting of extensions $\tau$ such that $\Pi_{\tau} \neq \emptyset$.
\end{cor}
\begin{proof} This follows from Lemma \ref{iota43} and Proposition \ref{modext1}.  \end{proof}

\begin{rem} Corollary \ref{surj1prime} does not imply that $\Pi_{\tau}\neq
\emptyset$ for all isomorphism classes $\tau \in \fT$. In fact, if we
replace $\bfF$ by its finite extension $\bfF'$ of degree $m$, then the order
of $\fT$ increases (since it is given by $\#H^1_{\Sigma}(F, \Hom(\tau_2,
\tau_1))/(q-1)$, where $q$ is the order of the residue field), while the
number of modular forms, i.e., $\sum_{\tau \in \fT} \# \Pi_{\tau}$ remains
the same. \end{rem}

\section{Bounding the size of $\prod_i\bfT_i/\phi(I_i)$} \label{Tsection}
In this section we keep in force Assumptions \ref{BK} and \ref{bound on T/J}. Moreover, we work in the two-dimensional setup and we set $\tau_1=1$ and $\tau_2=\chi$ (which can always be achieved by twisting by a Galois character). 
Let  $\mB:=\{e_1, \dots, e_s\}$ be a basis of $H^1_{\Sigma}(F, \Hom(\tau_2, \tau_1))=H^1_{\Sigma}(F, \chi^{-1})$ consisting of `modular' extensions, i.e., extensions $\tau$ such that $\Pi_{\tau}\neq \emptyset$ (cf. Corollary \ref{surj1prime}) and write $\tau^i$ for the corresponding residual representations. 
Let us write $\bfT_i$ for 
$\bfT_{\tau^i}$.  Similarly let us write $J_i$ and $\Pi_i$ for $J_{\tau^i}$ and $\Pi_{\tau^i}$ respectively.
Write $p_i: \bfT \twoheadrightarrow \bfT_i$ for the canonical projection. Consider the map $\bfT \rightarrow \prod_{i=1}^s \bfT_i.$ 
Let $J\subset \bfT$ be as in section \ref{The rings Ttau}.
 Set $J_i = p_i(J)$. Note that $J_i$ is an ideal because $p_i$ is surjective.

Let us begin with an observation that there is no reason to expect that 
 the canonical map $$\bfT/J \to \prod_{i=1}^s \bfT_i/J_i$$ should  in general be injective or surjective. In fact 
Lemma \ref{Tcyclic} shows that $\bfT/J$ is cyclic over $\Oo$. 
However, as we shall see below the orders of both sides are equal provided that the basis $\mB$ is unique up to scaling and that all of the ideals $J_i$ are principal.

\begin{prop} \label{l22} $\# \prod_{i=1}^s \bfT_i/J_i \leq \# \bfT/J$. \end{prop}

\begin{proof} In this proof we follow mostly the notation of \cite{Klosin09}, section 9. 
Let $\Pi_i$ be as before. As in  section \ref{lattice} we use Theorem \ref{Urban 1.1}  to get a lattice $\mL_i \subset \prod_{\pi \in \Pi_i}\rho_{\pi}$ and a finitely generated $\bfT_i$-module $\mT_i$ such that the following sequence of $\bfT_i/J_i[G_{\Sigma}]$-modules is exact: 
\be \label{ext43} 0 \to (\mT_i/J_i \mT_i)\otimes_{\Oo} \tilde{\tau}_1 \to \mL_i/J_i \to (\bfT_i/J_i)\otimes_{\Oo} \tilde{\tau}_2 \to 0.\ee 
Similarly to the situation in section \ref{lattice}, the sequence splits as a sequence of $\bfT_i$-modules, hence after tensoring with $\bfF$ we obtain a short exact sequence:
\be \label{ext43tensored} 0 \to (\mT_i/J_i \mT_i)\otimes_{\bfT_i} {\tau}_1 \to \mL_i/J_i\otimes_{\bfT_i} \bfF \to (\bfT_i/J_i)\otimes_{\bfT_i} {\tau}_2 \to 0.\ee 
Fix $i$. As in the proof of Lemma \ref{iota43} the sequence (\ref{ext43tensored}) gives rise to an injection 
 \be \label{iotai}\iota_{i}: \Hom(\mT_{i}/J_{i} \mT_{i}, \bfF) \hookrightarrow  H^1_{\Sigma}(F, \Hom_{\Oo}({\tau}_2, {\tau}_1)). \ee
 Arguing exactly as in the proof of Proposition \ref{modext1} we see that the image of $\iota_i$ is one-dimensional and is spanned by the cohomology classes corresponding to the isomorphism class of $\tau^i$, i.e., 
for every $i=1,2, \dots, s$ one has $\image (\iota_i) \subset \left< e_i \right>$. 
This implies that $ \Hom(\mT_i/J_i \mT_i,\bfF)$ is one-dimensional and hence $\mT_i/J_i \mT_i$ is a cyclic $\Oo$-module, and hence a cyclic $\bfT_i/J_i(=\Oo/\varpi^{d_i})$-module (cf. Lemma \ref{cyclicity}). On the other hand, again using Lemma 9.21 of \cite{Klosin09}, we get that $\textup{Fitt}_{\bfT_i}\mT_i=0$ and  this implies (as in the  proof of Lemma \ref{iota43J}) that 
$$\val_p(\# \bfT_i/J_i) \leq \val_p(\# \mT_i/J_i \mT_i).$$ This combined with the fact that $\mT_i/J_i\mT_i$ is a cyclic $\bfT_i/J_i$-module implies that $\bfT_i/J_i \cong \mT_i/J_i\mT_i$. In particular this implies  that the lattice $\mL_i/J_i \cong (\bfT_i/J_i)^2$ as $\bfT_i$-modules. 

Let $\tilde{\rho}_i: G \to \GL_2(\bfT_i/J_i)$ be the representation given by the short exact sequence $0 \to (\bfT_i/J_i)\otimes \tilde{\tau}_1 \to \mL_i/J_i \to (\bfT_i/J_i)\otimes \tilde{\tau}_2\to 0$ (coming from the sequence (\ref{ext43}) and the fact that $\mT_i/J_i\mT_i \cong \bfT_i/J_i$). One has $\bfT_i/J_i = \Oo/\varpi^{d_i}$ and since $\tilde{\rho}_i$ reduces to $\tau^i$ we must have $d_i \leq r_i$, where $\Oo c_i \cong \Oo/\varpi^{r_i} \Oo$. So, in particular we get that $\sum_i d_i \leq \sum r_i$. Combining Assumptions \ref{BK} with \ref{bound on T/J} we obtain the claim of Proposition \ref{l22}. \end{proof}

Our goal is now to prove the opposite inequality, which under some additional assumption will 
follow from a more general commutative algebra result which was proved by the authors and Kenneth Kramer in \cite{BergerKlosinKramer14} and which we will now present.

Let $s \in \bfZ_+$ and let $\{n_1, n_2, \dots, n_s\}$ be a set of $s$ positive integers. Set $n=\sum_{i=1}^s n_i$.
 Let $A_i=\Oo^{n_i}$ with $i\in \{1,2,\dots, s\}$.  Set $A=\prod_{i=1}^s A_i=\Oo^n$. Let $\varphi_i: A\twoheadrightarrow A_i$ be the canonical projection. Let $T \subset A$ be a (local complete) $\Oo$-subalgebra which is of full rank as an $\Oo$-submodule and let $J \subset T$ be an ideal of finite index. Set $T_i=\varphi_i(T)$ and $J_i=\varphi_i(J)$. Note that each $T_i$ is also a (local complete) $\Oo$-subalgebra and the projections $\varphi_i|_{T}$ are local homomorphisms. Then $J_i$ is also an ideal of finite index in $T_i$.
\begin{thm} [\cite{BergerKlosinKramer14}, Theorem 2.1] \label{Kr12} If $\# \bfF^{\times} \geq s-1$ and each $J_i$ is principal, then $\# \prod_{i=1}^s T_i/J_i \geq \# T/J$.\end{thm}

 Let $V$ be a vector space and write $\bfP^1(V)$ for the set of all lines in $V$ passing through the origin. There is a canonical map $V \setminus \{0\} \to \bfP^1(V)$ sending a vector $v$ to the line spanned by $v$.

Let $\mS$ be the set of all modular bases of $H^1_{\Sigma}(F, \Hom (\tau_2, \tau_1))$, i.e., the set of bases $\mB'=\{e'_1, e'_2, \dots, e'_s\}$ having the property that $\Pi_{\tau'_i} \neq \emptyset$, where $\tau'_i$ is the residual representation corresponding to the extension represented by $e'_i$. The set $\mS$ is non-empty as $\mB \in \mS$.

\begin{definition} \label{projuniq} We will say that $H^1_{\Sigma}(F, \Hom (\tau_2, \tau_1))$ has a \emph{projectively unique modular basis}  if the images of all the elements of $\mS$ in $\bfP^1(H^1_{\Sigma}(F, \Hom(\tau_2, \tau_1))$ agree. In the case when $H^1_{\Sigma}(F, \Hom (\tau_2, \tau_1))$ has this property we will refer to any element of $\mS$ as the projectively unique modular basis. \end{definition}

Note that it is possible to find $i_0 \in \{1,2, \dots, s\}$ such that the set $\mB':=\mB \cup \{ e_{\tau}\} \setminus \{e_{i_0}\}$ is still a basis of $H^1_{\Sigma}(F, \Hom(\tau_2, \tau_1))$ (and one still has that $\Pi_{\tau'}\neq \emptyset$ for all $\tau' \in \mB'$). Hence we can assume without loss of generality that $\mB=\{e_1, e_2, \dots, e_s\}$ with $\tau^1=\tau$. 
 In fact, if $H^1_{\Sigma}(F, \Hom (\tau_2, \tau_1))$ has a projectively unique modular basis, it follows that if $\mB'$ is another modular basis, then the isomorphism classes of the residual representations corresponding to the elements of $\mB'$ are the same as the isomorphism  classes of the residual representations corresponding to the elements of $\mB$.

\begin{prop} \label{otherine} Suppose $H^1_{\Sigma}(F, \Hom (\tau_2, \tau_1))$ has a projectively unique modular basis. If for each $i$, the ideal $J_i$ of $\bfT_i$ is principal and $\bfT_i/J_i$ is finite, then $\# \prod_{i=1}^s \bfT_i/J_i \geq \# \bfT/J$. \end{prop}

\begin{proof} First note that our assumption that $E$ be sufficiently large allows us to assume that $\#\bfF^{\times}$ satisfies the inequality in Theorem \ref{Kr12}. Since $\bfT$ is a free $\Oo$-module of finite rank we set $n$ to be that rank and define $n_i$ to be the rank of $\bfT_i$. The assumption that $\mB$ be projectively unique guarantees that every $\Oo$-algebra homomorphism has a corresponding residual Galois representation isomorphic to $\tau^i$ for some $i$. Hence $n =\sum_{i=1}^s n_i$.
Finally, note that $\bfT/J$ is finite. Indeed, first note that if we consider $\bfT$ as a (full rank) $\Oo$-subalgebra of $\prod_{i=1}^s \bfT_i$ then for a sufficiently large exponent $N$, we have $p^Ne_i \in \bfT$, where $e_i\in \prod_{i=1}^s\bfT_i$ is the idempotent corresponding to $\bfT_i$. On the other hand because $\bfT_i/J_i$ is finite for each $i$, there exists a positive integer $M$ such that $p^M \in J_i$ for each $i$. Let $x_i\in J$ be a preimage of $p^M \in J_i$. Then $p^{N+M} = \sum_{i=1}^s x_ip^N e_i \in J$, hence $T/J$ is torsion and thus finite. 
The Proposition now follows from Theorem \ref{Kr12} by taking $T=\bfT$. 
 \end{proof}

\begin{rem} \label{consequencemodular} Theorem \ref{Kr12} also has consequences for congruences between modular forms. Suppose that $\bfT=\bfT_{\Sigma}$ is the cuspidal Hecke algebra (acting on the space of automorphic forms over imaginary quadratic fields of weight 2 right invariant under a certain compact subgroup $K_f$) localized at a maximal ideal corresponding to an Eisenstein series, say $\mE$. Let $J=J_{\Sigma}$ be the Eisenstein ideal corresponding to $\mE$ (see section \ref{s43} for the details). Let $\mN$ be the set of $\Oo$-algebra homomorphisms $\bfT \to \Oo$, i.e. to cuspidal Hecke eigencharacters congruent to the eigencharacter $\lambda_0$ of $\mE$ mod $\varpi$. For $\lambda \in \mN$ write $m_{\lambda}$ for the largest positive integer such that $\lambda(T)\equiv \lambda_0(T)$ mod $\varpi^{m_{\lambda}}$ for all $T \in \bfT$. Let $e$ be the ramification index of $E$ over $\bfQ_p$. As a consequence of Theorem \ref{Kr12} we get the following inequality (cf. Proposition 4.3 in \cite{BergerKlosinKramer14} $$\frac{1}{e}\sum_{\lambda \in \mN} m_{\lambda} \geq \val_{\varpi}(\#\bfT/J).$$ 
For many more applications of Theorem \ref{Kr12} see [loc.cit.]. 
\end{rem}

\begin{cor} \label{exhaustive} Suppose that Assumptions  \ref{BK} and \ref{bound on T/J} hold. If the modular basis $\mB$ is projectively unique and 
if for each $i$, the ideal $J_i$ of $\bfT_i$ is principal and $\bfT_i/J_i$ is finite, 
then $$\# \prod_{i=1}^s \bfT_i/J_i = \# \bfT/J\geq \#\Oo/L.$$ \end{cor}

\section{Urban's method applied to $R^{\rm tr, 0}$} \label{urban}
In this section we again set $n=2$, $\tau_1=1$ and $\tau_2=\chi$ and we
fix a residual representation $$\tau: G_{\Sigma} \to \GL_{2}(\bfF), \quad \tau = \bmat 1 & \chi f \\ 0 & \chi \emat.$$

Let $\Phi_{\tau}, \Pi_{\tau}$ be as in section \ref{The rings Ttau}. 
From now on we will assume that $\Pi_{\tau}$ is non-empty.  Also recall that we make Assumption \ref{finiteness}, which by  Remark \ref{Phirem} is equivalent to assuming $\Phi_{\tau}$ is finite (in particular $R^0_{\tau}$ is defined and  finitely generated as an $\Oo$-module). 
The surjection $\phi: R_{\tau}^{\rm tr, 0} \twoheadrightarrow \bfT_{\tau}$  (cf. Definition \ref{def hecke} and Lemma \ref{factoring}) 
 descends to a surjection \be \label{modItr} R_{\tau}^{\rm tr, 0}/I_{\tau}^{\rm tr, 0} \twoheadrightarrow \bfT_{\tau}/J_{\tau}\ee (since by Lemma \ref{cyclicity2} $\phi_{\tau}(I^{\rm tr, 0}_{\tau}) = J_{\tau}$). 
The main goal of this section is to prove that under certain assumptions the map in (\ref{modItr}) is an isomorphism (Theorem \ref{modsurj}). Before we state the theorem let us demonstrate several properties of $R_{\tau}^{\rm tr,0}$. In particular we will show that $R_{\tau}^{\rm tr,0} \cong R_{\tau}^0$ (Theorem \ref{deformtoRtr}). In this section we also assume that Assumptions \ref{ass1}, \ref{BK} and \ref{bound on T/J} are satisfied.

\begin{lemma} \label{fractions} The $\Oo$-rank of $R_{\tau}^{0}$ equals the $\Oo$-rank of $R_{\tau}^{\rm tr, 0}$. In particular the normalizations (and the total rings of fractions) of $R_{\tau}^{0}$ and $R_{\tau}^{\rm tr, 0}$ coincide. \end{lemma}
\begin{proof} Write $\rho$ for $\rho_{\tau}^0$, i.e., $\rho: G_{\Sigma} \to \GL_2(R_{\tau}^{0})$.  We claim that we can conjugate $\rho$ so that for every $g \in G_{\Sigma}$ we have $\rho(g) = \bmat a & b \\ c & d \emat$ with $a,c,d \in R_{\tau}^{\rm tr, 0}$. Indeed, since the characters on the diagonal of $\tau$ are distinct mod $\varpi$, we can find $\sigma \in G_{\Sigma}$ on which they differ, so that the eigenvalues of $\tau(\sigma)$ lift by Hensel's lemma to
distinct eigenvalues of $\rho(\sigma)$ in $R_{\tau}^{0}$, and we can conjugate (over $R_{\tau}^{0}$) to have
 $\rho(\sigma)$ be diagonal with these lifted eigenvalues. For a
 general element $g\in G_{\Sigma}$, we then compare $\tr \rho(g)$ with $\tr \rho(\sigma g)$ and
use that the eigenvalues are distinct mod $\varpi$ to see that the two diagonal
 entries of $\rho(g)$ lie in $R_{\tau}^{\rm tr, 0}$. Similarly we show (cf. the proof of Lemma 3.27 in \cite{DDT}) that the lower-left entry also lies in $R_{\tau}^{\rm tr, 0}$. 

Note that since $R_{\tau}^{0}$ is a finitely generated $\Oo$-module 
(and $\Oo$ is assumed to be sufficiently large) we get an embedding $$R_{\tau}^{0} \hookrightarrow \hat{R}_{\tau}^{0} \cong  \prod_{i=1}^k \Oo,$$ where $\hat{R}_{\tau}^{0}$ is the normalization of $R_{\tau}^{0}$. For $i=1, \dots, k$, write $\rho_i$ for the composition of $\rho$ with the projection onto the $i$th component of $\hat{R}_{\tau}^{0}$. Suppose that the $\Oo$-rank of $R_{\tau}^{\rm tr, 0}$ is strictly smaller than the $\Oo$-rank of $R_{\tau}^{0}$. Then there exist two minimal primes (after possibly renumbering the minimal primes we will call them $\fp_1$, $\fp_2$) of $R_{\tau}^{0}$ which contract to the same minimal prime $\fp$ of $R_{\tau}^{\rm tr, 0}$.
Hence we get the following commutative diagram:
$$\xymatrix{R_{\tau}^{\rm tr, 0}\ar@{^{(}->}[r] \ar@{->>}[dd]&R_{\tau}^{0}\ar@{->>}[d] \ar@{->>}[dr]\\
&R_{\tau}^{0}/\fp_1 & R_{\tau}^{0}/\fp_2\\
R_{\tau}^{\rm tr, 0}/\fp\ar[ur]_{\sim}\ar[urr]_{\sim}}$$
This implies that the corresponding two deformations (to $\Oo$) $\rho_1 = \bmat a_1 & b_1 \\ c_1 & d_1 \emat$ and $\rho_2= \bmat a_2 & b_2 \\ c_2 & d_2 \emat$ (since their $a$-, $c$- and $d$-entries (as functions) factor through  $R_{\tau}^{\rm tr, 0}/\fp$) must satisfy $a_1=a_2=:a$, $c_1=c_2=:c$ and $d_1 = d_2=:d$. In particular their traces are equal. 
Using Lemma \ref{irr4} we see that both $\rho_1 \otimes_{\Oo} E$ and $\rho_2 \otimes_{\Oo} E$ are irreducible and thus by the Brauer-Nesbitt Theorem we conclude that $\rho_1 \otimes_{\Oo} E \cong \rho_2 \otimes_{\Oo} E$. Let $M =\bmat A & B \\ C & D \emat\in \GL_2(E)$ be such that \be \label{pro}M \rho_1 = \rho_2 M.\ee Then an easy matrix calculation shows that \be \label{firsteq} Aa+Bc = aA +b_2 C\ee and \be \label{secondeq} Cb_1 + Dd = cB +dD\ee from which we get that $Cb_1 = Cb_2$. Suppose for the moment that  there exists $g \in G_{\Sigma}$ such that $b_1(g) \neq b_2(g)$. Since $\Oo$ is a domain we conclude that $C=0$. 
Since the representations $\rho_1$ and $\rho_2$ are irreducible over $E$, the function $c$ cannot be identically zero. Using (\ref{secondeq}) we conclude that $B=0$. Finally, computing the lower-left entries on both sides of (\ref{pro}) we get $Dc=cA$, so again using the fact $c$ is not identically zero and that $\Oo$ is a domain we get that $A=D$. Thus, $M$ is a non-zero scalar matrix. Hence we get a contradiction to our assumption on the existence of $g$ and we conclude that $b_1=b_2$. In particular $\rho_1$ and $\rho_2$ are identical deformations of $\tau$ which correspond to distinct minimal primes of $R_{\tau}^{0}$. Hence $\rho_1$ and $\rho_2$ give rise to two different homomorphisms from $R_{\tau}^{\rm red}$ to $\Oo$. This contradicts the bijectivity of the correspondence $$\Hom_{\Oo-{\rm alg}}(R_{\tau}^{\rm red}, \Oo) \leftrightarrow \{\textup{deformations of $\tau$ into $\Oo$}\}.$$ \end{proof}

\begin{thm} \label{deformtoRtr}
There exists a deformation $\rho_{\tau}^{\rm tr, 0}: G_{\Sigma} \to \GL_2(R_{\tau}^{\rm tr, 0})$ of $\tau$. The resulting canonical map $R_{\tau}^{\rm red} \to R_{\tau}^{\rm tr, 0}$ factors through $R_{\tau}^0$ and induces an isomorphism $R_{\tau}^0 \cong R_{\tau}^{\rm tr, 0}$. \end{thm}

\begin{proof}
We will (once again) apply Theorem \ref{Urban 1.1} (due to Urban). In the notation of section \ref{lattice} we will write $\mF_{\mB} = \mF$ to be total ring of fractions of $\mB=R_{\tau}^{\rm tr, 0} \subset \mF$.
 Note that by Lemma \ref{fractions}, $\mF$ is also the total ring of fractions of $R_{\tau}^{0}$. Moreover, we take $\mR = R_{\tau}^{\rm tr, 0}[G_{\Sigma}]$, $\rho = \rho_{\tau} \otimes_{\mB} \mF: G_{\Sigma} \to \GL_{2}(\mF)$  which induces a morphism $\rho: R_{\tau}^{\rm tr, 0}[G_{\Sigma}] \to M_{2}(\mF)$ of $R_{\tau}^{\rm tr, 0}$-algebras. As before, the representations denoted in Theorem \ref{Urban 1.1} by $\rho_1$ and $\rho_2$ are our unique lifts $\tilde{\tau}_1,\tilde{\tau}_2: G_{\Sigma} \to \GL_{2}(\Oo) \hookrightarrow \GL_{2}(R_{\tau}^{\rm tr, 0})$ and we set $I=I_{\tau}^{\rm tr, 0}$. Note that conditions (i) and (ii) of Theorem \ref{Urban 1.1} are satisfied respectively by the definition of $R_{\tau}^{\rm tr, 0}$ and of $I_{\tau}^{\rm tr, 0}$ and conditions (iii) and (iv) are satisfied by our assumptions on $\tau_1$ and $\tau_2$. Finally, the condition of irreducibility of $\rho$ is satisfied by Lemma \ref{irr4}.

Hence we conclude from Theorem \ref{Urban 1.1} that there exists an $R_{\tau}^{\rm tr, 0}[G_{\Sigma}]$-stable lattice $\mL \subset \mF^{2}$ and a finitely generated $R_{\tau}^{\rm tr, 0}$-module $\mT_{\tau}\subset \mF$ 
such that we have an exact sequence of $R^{\rm tr, 0}_{\tau}[G_{\Sigma}]$-modules:
\be \label{I1} 0 \to \tilde{\tau}_1\otimes_{\Oo} \mT_{\tau}/I_{\tau}^{\rm tr, 0} \mT_{\tau} \to \mL \otimes_{R_{\tau}^{\rm tr, 0}} R_{\tau}^{\rm tr, 0}/I_{\tau}^{\rm tr, 0} \to \tilde{\tau}_2 \otimes_{\Oo} R_{\tau}^{\rm tr, 0}/I_{\tau}^{\rm tr, 0} \to 0. \ee

 It follows from \cite{Urban01} Lemmas 1.1 and 1.5 that $\mL = \mT_{\tau} \oplus  R_{\tau}^{\rm tr, 0}$ as $ R_{\tau}^{\rm tr, 0}$-modules. We will now show that $\mT_{\tau}\cong R_{\tau}^{\rm tr, 0}$. Indeed, the lattice $\mL$ is defined as in section \ref{lattice}, but since we only work with a fixed residual representation $\tau$, the representation $\rho$ in Theorem \ref{Urban 1.1} equals the $\prod_{\rho_{\pi}\in \Phi_{\tau}}\rho_{\pi}$. 
Using Lemma \ref{gene} for this representation (i.e., when we replace $\rho_{\Pi}$ with $\rho$ as above), we conclude that the $G_{\Sigma}$-module $\mL \otimes_{R_{\tau}^{\rm tr, 0}}\bfF$ is the subspace of $\prod_{\rho_{\pi} \in \Phi_{\tau}} \bfF^2$ 
generated by $\ov{\rho}(r)e_2$ 
 (with notation as in that lemma). This subspace is clearly isomorphic to $\tau$ as a $G_{\Sigma}$-module. So, the middle term in (\ref{I1}) after tensoring with $\bfF$ is two-dimensional, hence we must have $\mT_{\tau}/I_{\tau}^{\rm tr, 0} \mT_{\tau} \otimes_{R_{\tau}^{\rm tr, 0}} \bfF = \mT_{\tau}/\fm \mT_{\tau} = \bfF$, where $\fm$ is the maximal ideal of $R_{\tau}^{\rm tr, 0}$. Thus, by Nakayama's Lemma, we see that $\mT_{\tau}$ is generated over $R_{\tau}^{\rm tr, 0}$ by one element, say $x$. Consider the surjective map $\phi: R_{\tau}^{\rm tr, 0} \to \mT_{\tau}$ given by $r \mapsto rx$. Let $a\in \ker \phi$. Then $a$ annihilates $\mT_{\tau}$. However, by definition of $\mT_{\tau}$ and the fact that  $R_{\tau}^{\rm tr, 0}$ embeds into its ring of fractions $\mF$ we can consider $x$ and $a$ as elements of $\mF = \prod E$. If $a=(a_1, a_2, \dots, a_s)$ and $x=(x_1, x_2, \dots, x_s)$ and $a_j\neq 0$, then $x_j=0$. However, this implies that $\mT\otimes \mF \neq \mF$, which contradicts the fact that $\mT_{\tau}$ is a lattice in $\mF$ (cf. Theorem \ref{Urban 1.1}). 
Hence $\phi$ is injective and we indeed have $\mT_{\tau} = R_{\tau}^{\rm tr, 0}$. Thus $\mL = (R_{\tau}^{\rm tr, 0})^2$ as an $R_{\tau}^{\rm tr, 0}$-module and $\rho$ gives rise to a representation $\rho_{\tau}^{\rm tr, 0}: G_{\Sigma} \to \GL_2(R_{\tau}^{\rm tr, 0})$. 

By the above it is clear that $\rho_{\tau}^{\rm tr, 0}$ reduces to $\tau$. Let us make sure that the resulting representation is crystalline. Indeed, the lattice $\mL$ lives inside the finite direct product of the representations $\rho_{\pi}$ for $\pi \in \Phi_{\tau}$ and each of the $\rho_{\pi}$ is crystalline. Hence as a submodule of a finite direct sum of crystalline representations $\rho_{\tau}^{\rm tr, 0}$ is crystalline. This proves the first assertion.

By universality of $R_{\tau}^{\rm red}$ we obtain an $\Oo$-algebra map $\phi: R_{\tau}^{\rm red} \to R_{\tau}^{\rm tr, 0}$. Since $R_{\tau}^{\rm tr, 0}$ is a subring of the direct product of finitely many copies of $\Oo$, the map $\phi$ clearly factors through $R_{\tau}^0$. Let us denote the induced map $R_{\tau}^0 \to R_{\tau}^{\rm tr, 0}$ also by $\phi$. We claim that $\phi$ is surjective. Indeed, by its definition $R^{\rm tr, 0}$ is generated by traces of $\rho_{\tau}^{0}$. So, it is enough to show that the traces of $\rho_{\tau}^{0}$ and $\rho_{\tau}^{\rm tr, 0}=\phi \circ \rho^{0}_{\tau}$ coincide. This follows from the construction of the lattice $\mL$ which is a $G_{\Sigma}$-stable lattice inside the representation $\rho_{\tau}^{0} \otimes \mF$. 

In particular, both representations $\rho_{\tau}^{0}$ and $\rho_{\tau}^{\rm tr, 0}$ are isomorphic after tensoring with $\mF$, hence they must have equal traces. 
Since both $R_{\tau}^{0}$ and $R_{\tau}^{\rm tr, 0}$ are finitely generated $\Oo$-modules with the same rings of fractions (Lemma \ref{fractions}), the kernel of $\phi$ must be a torsion $\Oo$-module. This implies that the kernel must be zero (as $R_{\tau}^{ 0}$ embeds into $\prod E$). This proves the second assertion.
\end{proof}

Suppose that $\tau$ is such that $\Pi_{\tau} \neq \emptyset$. 
Recall that by Corollary \ref{surj1prime} there exists a basis $\mB=\{e_1, e_2, \dots, e_s\}$ of $H^1_{\Sigma}(F, \Hom(\tau_2, \tau_1))$ such that $\Pi_{\tau^i} \neq \emptyset$ for all $i=1,2,\dots, s$, where $\tau^i$ denotes the representation corresponding to the extension $e_i$.  It is possible to find $i_0 \in \{1,2, \dots, s\}$ such that the set $\mB':=\mB \cup \{ e_{\tau}\} \setminus \{e_{i_0}\}$ is still a basis (and one still has that $\Pi_{\tau'}\neq \emptyset$ for all $\tau' \in \mB'$). Hence we can assume without loss of generality that $\mB=\{e_1, e_2, \dots, e_s\}$ with $\tau^1=\tau$.

\begin{thm} \label{modsurj} Suppose Assumptions \ref{ass1}, \ref{BK}, \ref{finiteness} and \ref{bound on T/J} are satisfied. Suppose moreover that the modular basis $\mB$ is projectively unique and that for each $i=1,2, \dots, s$ the corresponding ideal $J_{\tau^i}$ is principal. Then the map
$$R_{\tau}^{\rm tr, 0}/I_{\tau}^{\rm tr, 0} \twoheadrightarrow \bfT_{\tau}/J_{\tau}$$ is an isomorphism. \end{thm}

\begin{proof}
 In the proof of Theorem \ref{deformtoRtr} we showed the existence of an $R_{\tau}^{\rm tr, 0}$-algebra morphism $\rho :R_{\tau}^{\rm tr, 0}[G_{\Sigma}] \to M_2(\mF)$ and an  $R_{\tau}^{\rm tr, 0}[G_{\Sigma}]$-stable lattice $\mL \subset \mF^2$ together with a finitely generated $R_{\tau}^{\rm tr, 0}$-submodule $\mT_{\tau} \subset \mF$ such that we have an exact sequence of $R^{\rm tr, 0}_{\tau}[G_{\Sigma}]$-modules:
\be  \label{neq1} 0 \to \tilde{\tau}_1\otimes_{\Oo} \mT_{\tau}/I_{\tau}^{\rm tr, 0} \mT_{\tau} \to \mL \otimes_{R_{\tau}^{\rm tr,0}} R_{\tau}^{\rm tr, 0}/I_{\tau}^{\rm tr, 0} \to \tilde{\tau}_2 \otimes_{\Oo} R_{\tau}^{\rm tr, 0}/I_{\tau}^{\rm tr, 0} \to 0 \ee 
which splits as a sequence of $R_{\tau}^{\rm tr, 0}/I_{\tau}^{\rm tr, 0}$-modules.
As in section \ref{lattice} we get a map:
$$\iota: \Hom_{\Oo}(\mT_{\tau}/I_{\tau}^{\rm tr, 0} \mT_{\tau}, E/\Oo) \to H^1(F_{\Sigma}, \Hom(\tilde{\tau}_2, \tilde{\tau}_1)\otimes_{\Oo} E/\Oo).$$ 
The fact that $\iota$ is injective and that its image is contained in the Selmer group is proved in the same way as Lemma \ref{iota43J}. Tensoring (\ref{neq1}) with $\bfF$  and arguing as in the proof of Proposition \ref{l22}  (this time using that $R_{\tau}^{\rm tr,0}/I_{\tau}^{\rm tr, 0}$ is cyclic by Remark \ref{r2.10}) we see that $\mT_{\tau}/I_{\tau}^{\rm tr, 0}\mT_{\tau}$ is cyclic and that 
\be \label{mT} \val_p(\#\mT_{\tau}/I_{\tau}^{\rm tr, 0} \mT_{\tau}) \geq \val_p(\#R_{\tau}^{\rm tr, 0}/I_{\tau}^{\rm tr, 0}).\ee

 Using the above arguments for the rings $R^{\rm tr, 0}_i$ corresponding to the residual representation arising from $e_i$, and putting them together we obtain 
\be \label{in76} \val_p(\#\prod_{i=1}^s R^{\rm tr, 0}_{i}/I^{\rm tr, 0}_{i}) \leq \val_p(\#\prod_{i=1}^s \mT_{i}/I^{\rm tr, 0}_{i}\mT_{i}) \leq \val_p(\# H^1_{\Sigma}(F, \Hom(\tilde{\tau}_2, \tilde{\tau}_1)\otimes E/\Oo))\ee (the first inequality comes from (\ref{mT})).
The assumption in Corollary \ref{exhaustive} that $\bfT_i/J_i$ are finite is satisfied since this is true for $R^{\rm tr, 0}_{i}/I^{\rm tr, 0}_{i}=R^0_{\tau}/I^0_{\tau}$ by Theorem \ref{deformtoRtr} and Lemma \ref{finiteq}.
Combining Assumptions \ref{BK}, \ref{bound on T/J} with 
Corollary \ref{exhaustive}
 we obtain that \be \label{in77} \val_p(\# H^1_{\Sigma}(F, \Hom(\tilde{\tau}_2, \tilde{\tau}_1)\otimes E/\Oo))
 \leq \val_p(\#\prod_{i=1}^s \bfT_{i}/J_{i}).\ee
Combining (\ref{in76}) with (\ref{in77}) we conclude that the maps $$R^{\rm tr, 0}_{i}/I^{\rm tr, 0}_{i} \twoheadrightarrow \bfT_{i}/J_{i}$$ must be isomorphisms for every $i=1,2,\dots, s$. This completes the proof of Theorem \ref{modsurj}. \end{proof}

\section{Imaginary quadratic case} \label{section4}

For the remainder of the article we specialize to the case $n=2$, $F$ imaginary quadratic, $\tau_1=1$ and $\tau_2=\chi: G_{\Sigma} \to \bfF^{\times}$ an anticyclotomic character 
(i.e., we assume that $\chi(cgc) = \chi(g^{-1})$ for all $g \in G_{\Sigma}$ and $c\in G_{\bfQ}$ a complex conjugation). 
 Recall again our assumption \ref{finiteness} that $\mP(\tau)$ is finite (so that $R^0_{\tau}$ is defined and finitely generated as an $\Oo$-module).

\subsection{Principality of the ideal of reducibility} \label{Principality section}

 Let $f \in H^1_{\Sigma}(F, \chi^{-1})$ be non-zero.  Set $\tau = \bmat 1 & \chi  f \\ & \chi \emat$ with universal deformation ring $R_{\tau}$. 
\begin{prop} \label{prin} Suppose Assumption \ref{ass1} is satisfied and write $\Psi$ for $\tilde{\tau}_2$. 
Then the ideal of reducibility $I_{\tau}^{0}$ of $R_{\tau}^{0}$ is principal.  \end{prop}
We begin with the following lemma.

\begin{lemma} \label{aux} There exists $\rho_{\tau}^{\rm 0,  opp}: G_{\Sigma} \to \GL_2(R^{0}_{\tau})$ such that $\ov{\rho}_{\tau}^{\rm 0, opp} = \bmat \chi & *\\  & 1\emat$ and is non-split. \end{lemma}
\begin{proof} Let $c$ be the complex conjugation. Define $\rho'_{\tau}$ by $$\rho'_{\tau}(g) = \rho_{\tau}^{0}(cgc).$$ Then $$\ov{\rho}'_{\tau} = \bmat 1 & \chi(cgc) f(cgc) \\ & \chi(cgc) \emat = \bmat 1 & \chi^{-1}(g) f(cgc) \\ & \chi^{-1}(g)\emat.$$ If this is split, then $f'\in H^1_{\Sigma}(F, \chi)$ defined by $f'(g) = f(cgc)$ is the zero cohomology class. However, $a \mapsto a^c$ gives an isomorphism of $H^1_{\Sigma}(F, \chi^{-1})$ onto  $H^1_{\Sigma}(F, \chi)$ (cf. Lemma 7.1.5 in \cite{Berger05}), so $f' \neq 0$ since $f \neq 0$. Now set $\rho_{\tau}^{\rm 0, opp} = \rho'_{\tau} \otimes \Psi$.
\end{proof}
\begin{proof} [Proof of Proposition \ref{prin}]
We first note that the ideal of reducibility of ${\rm tr} \rho_{\tau}^{0. \rm opp}$ 
equals $I_{\tau}^0$.
 By Proposition 3.1 in \cite{BergerKlosin13} we know that $\ker \rho^0_{\tau} = \ker \tr \rho^0_{\tau}$ and $\ker \rho_{\tau}^{0, \rm opp} = \ker \tr \rho_{\tau}^{0, \rm opp}$ 
since $R_{\tau}^{0}/I_{\tau}^{0}$ is finite by Lemma \ref{finiteq}.
Noting that $\rho^0_{\tau}( R_{\tau}^{0}[G_{\Sigma}])=\rho^{\rm 0, opp}_{\tau}( R_{\tau}^{0}[G_{\Sigma}])$ 
we therefore have an isomorphism $$R_{\tau}^{0}[G_{\Sigma}]/{\rm ker} \,  {\rm tr} \rho^0_{\tau} \cong R_{\tau}^{0}[G_{\Sigma}]/{\rm ker} \,  {\rm tr} \rho^{0, \rm opp}_{\tau}.$$ This means that applying \cite{BellaicheChenevierbook}  Proposition 1.7.4 with $\rho^0_{\tau}$ and $\rho^{0, \rm opp}_{\tau}$ establishes condition (1) in  \cite{BellaicheChenevierbook}  Proposition 1.7.5, so we can conclude that $I^0_{\tau}$ is principal. (Note that
the generic irreducibility assumption in the propositions in \cite{BellaicheChenevierbook} can be replaced by ${\rm ker} \rho= {\rm ker} T$). \end{proof}

\begin{cor} \label{Jprin} Suppose Assumptions \ref{ass1}, \ref{BK},  \ref{finiteness} and \ref{bound on T/J} are satisfied. 
Then the ideal $J_{\tau}$  of $\bfT_{\tau}$ is principal. \end{cor}
\begin{proof} Let $\phi_{\tau}: R_{\tau}^{\rm tr, 0} \twoheadrightarrow \bfT_{\tau}$ be the canonical surjection. By Corollary \ref{cyclicity2} we know that $\phi(I_{\tau}^{\rm tr, 0})=J_{\tau}$. The claim follows from combining Theorem \ref{deformtoRtr} with Proposition \ref{prin}. \end{proof}

\begin{rem}
For other fields $F$ (e.g. CM fields), principality of the ideal of reducibility would hold for conjugate self-dual representations  (see Theorem 2.11 of \cite{BergerKlosin13}). 
\end{rem}

\subsection{Selmer groups}
In this subsection we discuss Assumptions \ref{ass1} and \ref{BK} for certain characters $\chi$, for which we will prove our main results.

\begin{example} \label{Eimag}
Fix a prime $\fp$  lying over $p$ 
and denote by $i_{\fp}$ the fixed embedding $\ov{F} \hookrightarrow \ov{F}_{\fp}$ and $i_{\infty}: \ov{F} \hookrightarrow \bfC$. Let $\tau_2=\chi$ be a $p$-adic Galois character of the following form: Let
$\phi_1$, $\phi_2$ be two Hecke characters of infinity types $z$ and $z^{-1}$ respectively, and set $\phi=\phi_1/\phi_2$. Assume that $\Sigma$ contains the set $S_{\phi}$ of 
primes dividing $\fM_1 \fM_2 \fM_1^c \fM_2^c\textup{d}_Fp$, where $\fM_i$ denotes the conductor of $\phi_i$.  Let $\phi_{\fp}:G_{\Sigma} \to \Oo^*$ denote the
$\fp$-adic Galois character corresponding to $\phi$ defined by $\phi_{\fp}({\rm Frob}_{\fq})=i_{\fp}(i_{\infty}^{-1}(\phi(\varpi_{\fq})))$ for $\fq \notin \Sigma$. Set $\Psi:= \phi_{\fp} \epsilon$,
where $\epsilon$ is the $p$-adic cyclotomic character, and $\chi=\ov{\Psi}$. Assume also that if $\fq \in \Sigma$, then $\# k_{\fq} \not\equiv 1$ (mod $p$).
Under this assumption Assumption \ref{ass1} will be satisfied by Corollary 9.7 in \cite{BergerKlosin13}. 

 Let $L^{\rm int}(0,\phi)$ be the special $L$-value attached to $\phi$ as in \cite{BergerKlosin09}. Write $W$ for $\Hom_{\Oo}(\Psi, 1)\otimes E/\Oo$. In this case we adapt 
Assumption \ref{BK} as follows:
 \begin{conj} \label{BK1} $\# H^1_f(F, W) \leq \#
\Oo/\varpi^m$, where $m=\val_{\varpi}(L^{\tuint}(0, \phi))$.\end{conj}
Note that this conjecture implies Assumption \ref{BK} for $\Sigma=\Sigma_p=\{\fp \mid p\}$. However, our conclusions hold for all sets $\Sigma \supset \Sigma_p$ for which $H^1_{\Sigma} = H^1_f$ 
(see Lemma 5.6 of \cite{BergerKlosin13}).

\begin{rem} \label{rBK1} Conjecture \ref{BK1} can in many
cases be deduced from the Main conjecture proven by Rubin \cite{Rubin91}. If
$\phi^{-1}=\psi^2$ for $\psi$ a Hecke character associated to a CM elliptic
curve, then one can argue as follows. By Proposition 4.4.3 in \cite{Deepreprint} and using that $H^1_f(F, W) \cong H^1_f(F, W^c)$, we have
$\# H^1_f(F, W)= \#H^1_f(F,E/\Oo(\phi_{\fp}^{-1}))$.  Thus we can use Corollary 4.3.4 in
\cite{Deepreprint} which together with the functional equation satisfied by $L(0,\phi)$ implies the desired inequality. \end{rem}

\end{example}

\subsection{Link of rings $\bfT_{\tau}$ to an actual Hecke algebra} \label{s43} In this section we recall from \cite{B09} an Eisenstein ideal bound for a Hecke algebra $\bfT(\Sigma)$ acting on cuspidal automorphic forms. We also recall results about Galois representations associated to such forms and use this to  relate $\bfT(\Sigma)$ to the ring $\bfT$ defined in Section \ref{The rings Ttau}.

Continuing with the notation of Example \ref{Eimag} we assume $\phi=\phi_1/\phi_2$ is unramified. For an ideal
$\mathfrak{N}$ in $\mathcal{O}_F$ and a finite place ${\fq}$ of $F$
put $\mathfrak{N}_{\fq}=\mathfrak{N} \mathcal{O}_{F,{\fq}}$. We
define
$$U^1(\mathfrak{N}_{\fq})=\{k \in {\rm GL}_2(\mathcal{O}_{F,{\fq}})\mid
{\rm det}(k) \equiv 1 \mod{\mathfrak{N}_{\fq}} \}.$$
Now put
\be
\label{ourcompact} K_{f}:=\prod_{{\fq} \mid \fM_1}
U^1(\mathfrak{M}_{1,{\fq}}) \subset
\Res_{F/\bfQ} \GL_2(\bfA_{f}).\ee
Denote by $S_2(K_{f})$ the space of
cuspidal
automorphic forms of $\Res_{F/\bfQ} \GL_2(\AQ)$ of weight 2, right-invariant under
$K_{f}$ (for more details see Section 3.1 of \cite{Urban95}). Put $\gamma=\phi_1 \phi_2$ and
write $S_2(K_{f}, \gamma)$ for
the forms with central character $\gamma$.

From now on, let $\Sigma$ be a finite set of places of $F$ containing
$$S_{\phi}:=\{{\fq} \mid \fM_1 \fM_1^c \} \cup \{{\fq}
\mid p d_F\}.$$

We denote by $\bfT(\Sigma)$ the $\Oo$-subalgebra of
$\End_{\Oo}(S_2(K_{f},\gamma))$
generated by the Hecke operators
$T_{\fq}$
for all places $\fq \not\in \Sigma$.

Let $J(\Sigma) \subset \bfT(\Sigma)$ be the ideal generated by
$$\{T_{\fq} - \phi_1(\varpi_{\fq}) \cdot \# k_{\fq}-
\phi_2(\varpi_{\fq}) \mid \fq \not\in \Sigma \}.$$

\begin{definition}
Denote by
$\fm(\Sigma)$ a
maximal ideal of $\bfT(\Sigma)$ containing the image of
$J(\Sigma)$. We set $\bfT_{\Sigma}:=
\bfT(\Sigma)_{\fm(\Sigma)}$. Moreover, set $J_{\Sigma}:=
J(\Sigma) \bfT_{\Sigma}$. We refer to $J_{\Sigma}$ as the
\emph{Eisenstein ideal of $\bfT_{\Sigma}$}.
\end{definition}

\begin{thm} [\cite{B09}
Theorem 14] \label{Eiscong} Let $p>3$ and assume $\ell \not \equiv \pm 1 \mod{p}$ for $\ell \mid d_F$. Let $\phi$ be an unramified Hecke
character of infinity type $\phi^{(\infty)}(z)=z^2$. There exist
Hecke characters $\phi_1, \phi_2$ with $\phi_1/\phi_2=\phi$ such
that their conductors are divisible only by ramified primes or inert
primes not congruent to $\pm 1 \mod{p}$ and such that $$\#
(\bfT_{\Sigma}/J_{\Sigma}) \geq \#(\Oo/(L^{\rm int}(0,\phi))).$$
\end{thm}

The space $S_2(K_{f}, \psi)$ is isomorphic as a
$G(\AQ_f)$-module to $\bigoplus \pi_{f}^{K_{f}}$ for automorphic
representations $\pi$ of a certain infinity type (see Theorem \ref{attach1}
below)
with central character $\psi$. Here $\pi_{f}$ denotes the restriction
of $\pi$ to $\GL_2(\AQ_f)$ and $\pi_{f}^{K_{f}}$ stands for the
$K_{f}$-invariants.

For $g \in G(\AQ_f)$ we have the usual Hecke
action of $[K_{f} g K_{f}]$ on $S_2(K_{f})$ and $S_2(K_{f},
\psi)$. For primes $\fq$ such that the $v$th component of $K_f$ is
$\GL_2(\OFv)$
we define $T_{\fq}
= [K_{f} \bmat \varpi_{\fq} \\ & 1 \emat K_{f}] $.

Combining the work of Taylor, Harris, and Soudry with results of
Friedberg-Hoffstein and Laumon/Weissauer, one can show the following
(see \cite{BergerHarcos07} for general case of cuspforms of weight
$k$ and forthcoming work for general CM fields by C. P. Mok (with a similar assumption on the central character) and Harris-Taylor-Thorne-Lan  (without such an assumption)):
\begin{thm}[\cite{BergerHarcos07} Theorem 1.1] \label{attach1} Given
a
cuspidal automorphic representation
$\pi$ of
$\GL_2(\AF)$ with $\pi_{\iy}$ isomorphic to the principal series
representation corresponding to $$\bmat t_1 & * \\ & t_2 \emat \mapsto
\left(\frac{t_1}{|t_1|} \right) \left(\frac{|t_2|}{t_2}\right)$$ and
cyclotomic central character
$\psi$ (i.e., $\psi^c = \psi$), let $\Sigma_{\pi}$ denote the set
consisting of the places of $F$ lying above $p$, the primes where $\pi$ or
$\pi^c$ is ramified, and the primes ramified in $F/\bfQ$.

Then there exists a finite extension $E$ of $F_{\fp}$ and a Galois
representation $$\rho_{\pi}: G_{\Sigma_{\pi}} \rightarrow \GL_2(E)$$
such that if $\fq \not\in \Sigma_{\pi}$, then $\rho_{\pi}$ is
unramified at $\fq$ and the characteristic polynomial of
$\rho_{\pi}(\Frob_{\fq})$ is $x^2-a_{\fq}(\pi)x + \psi(\varpi_{\fq})
(\# k_{\fq}),$ where $a_{\fq}(\pi)$ is the Hecke eigenvalue
corresponding to $T_{\fq}$. Moreover, $\rho_{\pi}$ is absolutely
irreducible.
\end{thm}

Regarding the crystallinity of the representations $\rho_{\pi}$ we make the following conjecture (see Section \ref{s2.1} for the definition of a short crystalline Galois representation, 
and note that we assume $p>3$):
\begin{conj} \label{chrys}
If $\pi$ is unramified at $\fq \mid p$ then
$\rho_{\pi}|_{G_{\mathfrak{q}}}$ is crystalline and short.
\end{conj}
This has now been proven in many cases by A. Jorza \cite{Jorza10}. Note that for the choice of characters $\phi_i$ as in Theorem \ref{Eiscong} the cuspforms occurring in $S_2(K_{f}, \psi)$ are unramified at $\fq \mid p$.

\begin{definition} \label{Pitau im} Let $\chi$ be the mod $\varpi$ reduction of $\phi_{\fp} \epsilon$. We now define the subsets $\Pi$ as the set of (strict equivalence classes of Galois) deformations of residual representations of the form (\ref{form1}), one for each $\rho_{\pi}$ associated to an automorphic representation $\pi$ occurring in $S_2(K_f, \gamma)_{\mathfrak{m}(\Sigma)}$
and define $\Pi_{\tau}$ to be the subset with residual representation isomorphic to (a twist by $\phi_{2,\fp}$ of) $\tau=\bmat 1 & * \\ 0 & \chi\emat$. (Note that $\Pi_{\tau} \cap \Pi_{\tau'}=\emptyset$ for $\tau \not \cong \tau'$ by Proposition \ref{SW2.13}).
\end{definition}

For every $\tau$ one has the natural surjective map $\bfT \twoheadrightarrow \bfT_{\tau}$ resulting from the surjections $R^{\rm ps} \twoheadrightarrow R^{\rm tr}_{\tau}$.

\begin{lemma} \label{actual Hecke} If $\Pi$ is the set of modular deformations defined above then the ring $\bfT \subset \prod_{\rho_{\pi} \in \Pi} \Oo$ defined in the previous section can be identified with the  Hecke algebra $\bfT_{\Sigma}$.
Furthermore, $\bfT_{\tau}$ agrees with the quotient of $\bfT_{\Sigma}$ acting on the subspace of automorphic forms spanned by eigenforms $\pi$ such that $\rho_{\pi} \in \Pi_{\tau}$. \end{lemma}

\begin{proof} We will just prove the first part (concerning $\bfT$ and $\bfT_{\Sigma}$ - the proof for $\bfT_{\tau}$ being analogous).
We define the following $\Oo$-algebra map:
$$f:\bfT_{\Sigma} \to \bfT \subset \prod_{\rho_{\pi} \in \Pi} \Oo: \; T_{\fq} \mapsto (a_{\fq}(\pi))_{\rho_{\pi} \in \Pi},$$ where we use that $a_{\fq}(\pi)=\tr \rho_{\pi}({\rm Frob}_{\fq})$ and therefore  $(a_{\fq}(\pi))_{\rho_{\pi} \in \Pi}= T(\Frob_{\fq})$. 

We check that this map $f$ is injective: By definition, $\mathbf{T}_{\Sigma}
\hookrightarrow \oplus_{\rho_{\pi} \in \Pi} {\rm End}_{\Oo} (V_{\pi}^{K_f})$,
where we denote by $V_{\pi}$ the representation space of $\pi$. 
Since $T_v$ acts on $\pi$ by $a_v(\pi)$, the image in each
summand is given by the $\Oo$-algebra generated by the
$a_v(\pi)$'s.
 Hence injectivity follows.

For surjectivity first note that $f(\bfT_{\Sigma}) \supset S:=\{ T(\Frob_{\fq}) \mid \fq \not\in \Sigma\}$. Since $f$ is injective let us identify $\bfT_{\Sigma}$ with $f(\bfT_{\Sigma})$. Clearly  $\bfT_{\Sigma}$ is local, complete with respect to its maximal ideal $\fm_{\Sigma}:=\fm(\Sigma)\bfT_{\Sigma}$ and $\bfT$   has the same properties derived from the properties of $R^{\rm ps}$. Moreover, looking at the residue fields  we see that $\fm_{\bfT} \cap \bfT_{\Sigma} = \fm_{\Sigma}$, so the $\fm_{\Sigma}$-adic topology on $\bfT_{\Sigma}$ is the subspace topology induced from the $\fm_{\bfT}$-adic topology on $\bfT$. Then using Theorem 8.1 in \cite{Matsumura} again we see that the closure of $\bfT_{\Sigma}$ in $\bfT$ equals the completion of $\bfT_{\Sigma}$. But since $\bfT_{\Sigma}$ is already complete and $\bfT$ is topologically generated by $S$ (i.e., the closure of $S$ in $\bfT$ equals $\bfT$), we conclude that $\bfT_{\Sigma} = \bfT$, hence we are done.   
\end{proof}

\section{Main result} \label{Main result}

In this section we will state our main theorems (Theorems \ref{mainthm}, \ref{mainthm2} and \ref{version2}) for the two-dimensional Galois representations over imaginary quadratic fields considered in section \ref{section4}. In this case many of the assumptions introduced throughout the paper can be proven to hold. However, we would like to stress that the conclusions are still valid if instead one assumes Assumptions \ref{ass1}, \ref{BK} and \ref{bound on T/J}.  To make this section self-contained we will repeat all the assumptions in the case of an imaginary quadratic field which were made in section \ref{section4}.

Let $F$ be an imaginary quadratic field,  $p>3$, $p\nmid \#{\rm Cl}_F d_F$, and assume $\ell \not \equiv \pm 1 \mod{p}$ for $\ell \mid d_F$. 
Let $\phi$ be an unramified Hecke character of infinity type $\phi^{(\infty)}(z)=z^2$ and write $\chi$ for the mod $\varpi$-reduction of $\phi_{\fp}\epsilon$. Furthermore, assume that $\Sigma$ contains the set of places $S_{\phi}$ (containing the primes dividing the conductors of the two auxiliary characters $\phi_i$ from Theorem \ref{Eiscong}).

Let $\tau: G_{\Sigma} \to \GL_2(\bfF)$ be a 
 non-semi-simple representation of the form
 $$\tau = \bmat 1 & * \\ 0 & \chi \emat.$$
\begin{thm}\label{mainthm} Suppose Conjecture \ref{BK1} is satisfied (see Remark \ref{rBK1}). There exists an $\bfF$-basis $\mB$ of $H^1_{\Sigma}(F, \chi^{-1})$ which is modular, i.e., such that if $b \in \mB$ and $f: G_{\Sigma} \to \bfF(\chi^{-1})$ is a cocycle representing $b$, then the residual representation $$\rho_f: G_{\Sigma}  \to \GL_2(\bfF), \quad \rho_f(g) = \bmat 1 & f(g) \chi(g) \\ 0 & \chi(g)\emat$$ is (up to a twist) the reduction (mod $\varpi$) of a representation $\rho_{\pi}: G_{\Sigma} \to \GL_2(\Oo)$ attached to an automorphic representation $\pi$ of $\GL_2(\AF)$.\end{thm}
\begin{proof} This follows from Corollary \ref{surj1prime}. Note that Assumption \ref{bound on T/J} is satisfied by Theorem \ref{Eiscong} 
and Lemma \ref{actual Hecke} (which implies that $J \subseteq J_{\Sigma}$ since $T \mod{J_{\Sigma}}$ is the sum of two characters; in fact, $J =J_{\Sigma}$ by the proof of Lemma \ref{Tcyclic}) 
and Conjecture \ref{BK1} replaces Assumption \ref{BK} in Corollary \ref{surj1prime}. \end{proof}

From now on assume that Conjecture  \ref{BK1} is satisfied. Then by Theorem \ref{mainthm} the Selmer group $H^1_{\Sigma}(F, \chi^{-1})$ has a modular basis $\mB$. From now on assume also that $\mB$ is indeed projectively unique. Then, as discussed in section \ref{Tsection}, the set of modular residual representations (i.e., those residual representations $\tau'=\bmat 1&* \\ 0 & \chi\emat$ for which $\Pi_{\tau'}\neq\emptyset$) is in one-to-one correspondence with the elements of $\mB$ given by sending a modular extension to the corresponding residual representation. In particular, the extension corresponding to $\tau$ must be (up to scaling) among the elements of $\mB$. Hence, by rescaling one of the elements of $\mB$ and reordering we may assume that $\mB = \{e_1=e_{\tau}, e_2, \dots, e_s\}$.

We assume further now that if $\fq \in \Sigma$, then $\# k_{\fq} \not\equiv 1$ (mod $p$). 
Using the principality of the ideal of reducibility (see Proposition \ref{prin}) we can prove a modularity result for residually reducible representations:

\begin{thm} \label{mainthm2}
Assume that the Galois representations $\rho_{\pi}$ for $\pi$ occurring in $S_2(K_f,\Psi)_{\fm(\Sigma)}$ (for notation please see Section \ref{s43}) are crystalline at $v\mid p$. Also suppose that $\# H^1_f(F, E/\Oo((\phi_{\fp}\epsilon))^{-1}) \leq \#
\Oo/\varpi^m$, where $m=\val_{\varpi}(L^{\tuint}(0, \phi))$.

Let $\rho: G_{\Sigma} \to \GL_2(E)$ be a continuous, irreducible representation which is crystalline at $\fp \mid p$ and write $\tau: G_{\Sigma} \to \GL_2(\bfF)$ for its mod $\varpi$ reduction with respect to some lattice in $E^2$. 
Suppose that $\tau^{\rm ss} \cong 1 \oplus \chi$. Assume that 
the sets $\Phi_{\tau'}$ for $\tau' \in \mB$ are finite. Then  $\rho$ is modular, i.e., there exists an automorphic representation $\pi'$ of $\GL_2(\AF)$ such that $\rho \cong \rho_{\pi'}$.
\end{thm}

\begin{rem}
As discussed in Remark \ref{rBK1} and the paragraph following Conjecture \ref{chrys} the first two assumptions should be satisfied in the majority of cases by work of Jorza and Rubin.
\end{rem}

\begin{rem} 
We note  that when ${\rm rk}_{\Oo} \bfT={\rm dim}_{\bfF} H^1_{\Sigma}(F, \Hom (\tau_2, \tau_1))$ it is easy to see that $H^1_{\Sigma}(F, \Hom (\tau_2, \tau_1))$ has a projectively unique modular basis. It may be possible to check  numerically in specific examples if ${\rm rk}_{\Oo} \bfT={\rm dim}_{\bfF} H^1_{\Sigma}(F, \Hom (\tau_2, \tau_1))$ holds, but we have unfortunately not been able to carry this out.
\end{rem}

\begin{proof} This is a summary of the arguments carried out so far. As in the proof of Theorem \ref{mainthm} we note that Assumption \ref{bound on T/J} is satisfied. Also Assumption \ref{ass1} is satisfied by Corollary 9.7 in \cite{BergerKlosin13} (see also discussion in Example \ref{Eimag}). 
 Let $R^{\rm tr, 0}_i$, $\bfT_i$ be as before the $\Oo$-subalgebra of $R_{\tau_i}^{0}$ (defined as the image of $R_{\tau_i}^{\rm red}$ inside $\prod_{\fp \in \mathcal{P}_{\tau_i}} \Oo$) generated by traces and the Hecke algebra (respectively) corresponding to $e_i$. We denote the corresponding ideals of reducibility and the Eisenstein ideal by $I^{\rm tr, 0}_i$ and $J_i$ respectively. We get for every $i=1,2,\dots, s$ a commutative diagram with surjective arrows:
\be \label{104} \xymatrix{R_i^{\rm tr, 0} \ar[r]^{\phi_i}\ar[d] & \bfT_i \ar[d] \\ R^{\rm tr, 0}_i/I^{\rm tr, 0}_i \ar[r]& \bfT_i/J_i.}\ee By Theorem \ref{modsurj} 
 and Corollary \ref{Jprin} the bottom map is an isomorphism for every $i=1,2, \dots, s$.

By Theorem \ref{deformtoRtr} we obtain a canonical map $R_{\tau}^{\rm red} \to R_{\tau}^{\rm tr, 0}$ which factors through an isomorphism $R_{\tau}^{0} \cong R_{\tau}^{\rm tr, 0}$. 
We know by Proposition \ref{prin} that $I_{\tau}^{0}$ is principal.
Using 
this and diagram (\ref{104}) for $i=\tau$ (and our conclusion that the bottom map in (\ref{104}) is an isomorphism) 
we can apply Theorem 4.1 from  \cite{BergerKlosin13} to 
conclude that $\phi_{\tau}$ is an isomorphism. Then any $\Oo$-algebra map $R_{\tau} \twoheadrightarrow \Oo$ factors through $R_{\tau}^{0} \cong \bfT_{\tau}$, so any deformation to $\Oo$ is modular.
\end{proof}

Let $\Psi$ be as in example \ref{Eimag}.
The property of $H^1_{\Sigma}(F, \chi^{-1})$ of having a projectively unique modular basis can be replaced by the assumption that $H^1_{\Sigma}(F, \Psi^{-1}\otimes E/\Oo)$ is an $\bfF$-vector space.
\begin{thm} \label{version2} Let the notation and assumptions be the same as in Theorem \ref{mainthm2} except that we do not demand that $H^1_{\Sigma}(F, \chi^{-1})$ has a projectively unique modular basis, but instead assume that $\varpi H^1_{\Sigma}(F, \Psi^{-1}\otimes E/\Oo)=0$. Then, as before, $\rho$ is modular. \end{thm}
\begin{proof} We need to reprove Theorem \ref{modsurj}. Our assumption that $\varpi$ annihilates the Selmer group along with  injectivity of $\iota$ in the proof  of Theorem \ref{modsurj} implies that $\varpi$ annihilates $\mT_{\tau}/I_{\tau}^{\rm tr, 0} \mT_{\tau}$. We showed in the proof of Theorem \ref{deformtoRtr} that $\mT_{\tau} \cong R_{\tau}^{\rm tr,0}$. Hence 
we conclude that $\varpi$ annihilates $R_{\tau}^{\rm tr, 0}/I_{\tau}^{\rm tr, 0}$. As discussed towards the end of section \ref{r0section}, the module $R_{\tau}^{\rm tr, 0}/I_{\tau}^{\rm tr, 0}$ is a cyclic $\Oo$-module, hence we must have $R_{\tau}^{\rm tr, 0}/I_{\tau}^{\rm tr, 0} \cong \bfF$.  Since $\bfT_{\tau}/J_{\tau}$ is a non-zero $\Oo$-module, this implies that the map $$R_{\tau}^{\rm tr, 0}/I_{\tau}^{\rm tr, 0} \twoheadrightarrow \bfT_{\tau}/J_{\tau}$$ must be injective. 
\end{proof}


\bibliographystyle{amsalpha}
\bibliography{standard2}

\end{document}